\documentclass[11pt]{amsart}
\usepackage[letterpaper,margin=1.2in]{geometry}
\usepackage{amsbsy,amsfonts,amsmath,amssymb,amscd,amsthm,mathrsfs}
\usepackage{subfigure,enumitem,graphicx,epsfig,color}
\usepackage[colorlinks=true,linkcolor=blue,citecolor=green]{hyperref}
\normalsize

\newtheorem{theorem}{Theorem}[section]
\newtheorem{lemma}[theorem]{Lemma}

\newtheorem{definition}[theorem]{Definition}
\newtheorem{corollary}[theorem]{Corollary}
\newtheorem{remark}[theorem]{Remark}

\newtheorem{proposition}[theorem]{Proposition}

\newtheorem*{theorem A}{Theorem A}
\newtheorem*{corollary B}{Corollary B}
\newtheorem*{corollary C}{Corollary C}
\newtheorem*{theorem B}{Theorem B}
\theoremstyle{definition}
\newtheorem{example}[theorem]{Example}

\begin{document}

\title[On topological entropy and topological pressure of non-autonomous IFSs]{On topological entropy and topological pressure of non-autonomous iterated function systems}
\author[F. H. Ghane]{F. H. Ghane$^{*}$}
\author[J. Nazarian Sarkooh]{J. Nazarian Sarkooh}
\address{Department of Mathematics, Ferdowsi University of Mashhad, Mashhad, IRAN.}
\email{\textcolor[rgb]{0.00,0.00,0.84}{javad.nazariansarkooh@gmail.com}}
\email{\textcolor[rgb]{0.00,0.00,0.84}{ghane@math.um.ac.ir}}
 \thanks{$^*$Corresponding author}
\begin{abstract}
In this paper we introduce the notions of topological entropy and topological pressure for non-autonomous iterated function systems (or NAIFSs for short) on countably infinite alphabets. NAIFSs differ from the usual (autonomous)
iterated function systems, they are given \cite{LGMU} by a sequence of collections of continuous
maps on a compact topological space, where maps are allowed to vary between iterations.
Several basic properties of topological pressure and topological entropy of NAIFSs are provided. Especially, we generalize the classical Bowen's result to NAIFSs ensures that the topological entropy is concentrated on the set of nonwandering points. Then, we define the notion of specification property, under which, the NAIFSs have positive topological entropy and all points are entropy points. In particular, each NAIFS with the specification property is topologically chaotic. Additionally, the $\ast$-expansive property for NAIFSs is introduced. We will prove that the topological pressure of any continuous potential can be computed as a limit at a definite size scale whenever the NAIFS satisfies the $\ast$-expansive property.
Finally, we study the NAIFSs induced by expanding maps. We prove that these NAIFSs
having the specification and $\ast$-expansive properties.
\end{abstract}

\maketitle
\thispagestyle{empty}
\section{Introduction}
The time dependent systems so-called non-autonomous, yield very flexible models than autonomous cases
for the study and description of real world processes. They may be used to describe the evolution of a wider
class of phenomena, including systems which are forced or driven. Non-autonomous dynamical systems are
strongly motivated from applications, e.g., in population biology \cite{RHS} as well as applications to numerical approximations, switching systems \cite{Kl} and synchronization \cite{Kl1}. Here, we deal with
non-autonomous iterated function systems (or NAIFSs for short) which differ from the usual
(autonomous) iterated function systems. It is natural, and frequently necessary in applications,
to consider the non-autonomous version of iterated function systems, where the system is allowed to
vary at each time. (In the case where all maps are affine similarities, the resulting system is also called a Moran
set construction \cite{LGMU}). Generalized Cantor sets that studied by Robinson and Sharples \cite{RS} are
examples of attractors of NAIFSs. Olson et al. \cite{ORS} illustrate examples of pullback attractors. A pullback attractor serves as non-autonomous counterpart to the global attractor. Henderson et al. \cite{AMORS},
extended the regularity results of \cite{ORS} to a natural class of attractors of both autonomous and
non-autonomous iterated functions systems of contracting similarities, and studied the Assouad, box-counting, Hausdorff and packing dimensions for the attractors of these class of dynamical systems. These regularity results are useful as pullback attractors can exhibit dimensionally different behaviour at different length scales. Rempe-Gillen and Urba\'{n}ski \cite{LGMU}, studied the Hausdorff dimension of the limit set of NAIFSs. Under a suitable restriction on the growth of the number of contractions used at each step, they showed that the Hausdorff dimension of the limit set is determined by an equation known as Bowen’s formula.
Also, they proved Bowen’s formula for a class of infinite alphabet systems and deal with Hausdorff measures for finite systems, as well as continuity of topological pressure and Hausdorff dimension for both finite and infinite systems. In particular they strengthened the existing continuity results for infinite autonomous systems.

In general, an NAIFS generalizes the both concepts of finitely generated semigroups and non-autonomous discrete dynamical systems. Recently, there have been major efforts in establishing a general theory of NAIFSs \cite{AMORS,LGMU}, but a global theory is still out of reach.
Our main goal in this paper is to describe the topological aspects
of thermodynamic formalism for NAIFSs.
To our knowledge, the thermodynamic formalism of such systems (NAIFSs) have not been studied before.
From a conceptual point of view, an interesting aspect of these studies is the fact that the
fundamental notions of thermodynamic formalism, like topological entropy and topological pressure, come up naturally in our context. However, an extension of the thermodynamical formalism for NAIFSs has revealed fundamental difficulties.

Thermodynamic formalism, i.e. the formalism of equilibrium statistical physics, was adapted to the
theory of dynamical systems in the classical works of Sinai, Ruelle and Bowen \cite{RB3,RBDR,DR1,YGS}.
Topological pressure and topological entropy are two fundamental notions in thermodynamic formalism.
Topological pressure is the main tool in studying dimension of invariant sets and measures for dynamical
systems in dimension theory. On the other hand, the notion of entropy is one of the most important objects in dynamical systems, either as a topological invariant or as a measure of the chaoticity of dynamical systems.
Hence, there were several attempts to find their generalization for other systems in an attempt to describe their dynamical characteristics, see, for instance, \cite{HWZ,WHYY,SKLS,DMMW,DJT1,FZKYGH}.

The concept of topological entropy of a map plays a central role in topological dynamics. There are two
standard definitions of topological entropy for a continuous self-map of a compact metric space \cite{BHZNJP}.
The first definition was given by Adler, Konhelm and McAndrew \cite{AKM}, based on open covers, can be
applied to continuous maps of any compact topological space. In 1971, Bowen \cite{RB} and Dinaburg \cite{Di} gave other definitions, based on the dispersion of orbits, for uniformly continuous maps in metric spaces.
When the metric space is compact, these definitions yield the same quantity, which is an invariant of topological conjugacy. Also, Bowen \cite{RB1} gave a characterization of dimension type for topological entropy of
non-compact and non-invariant sets. Topological entropy has close relationships with many important dynamical properties, such as chaos, Lyapunov exponents, the growth of the number of periodic points and so on.
Moreover, positive topological entropy has remarkable role in the characterization of the dynamical behaviors,
for instance, Downarowicz proved that positive topological entropy implies chaos DC2 \cite{TG}. Thus, a lot of attention has been focused on computations and estimations of topological entropy of an autonomous dynamical system and many good results have been obtained \cite{LBWC,RB2,RB,LWG}.
Beyond autonomous dynamical systems, several authors provided conditions for computations and estimations
of topological entropy, for instance, Shao et al. \cite{SSZ} have given an estimation of lower bound of topological entropy for coupled-expanding systems associated with transition matrices in compact Hausdorff spaces.
Some knowledge of topological entropy of semigroup actions is also available in \cite{AB1, ABMU, FRPV}.

The notion of specification was introduced in the seventies as a property of uniformly
hyperbolic basic pieces and became a characterization of complexity in dynamical systems.
Rodrigues and Varandas \cite{FRPV} introduced some notions of specifcation for semigroup actions
and proved that any finitely generated continuous semigroup action on a compact metric space with the strong orbital specification property has positive topological entropy; moreover, every point is an entropy point. Roughly speaking, entropy points are those that their local neighborhoods reflect the complexity
of the entire dynamical system from the viewpoint of entropy theory. Also, these results extended to non-autonomous discrete dynamical systems \cite{JNFG}. In the current paper, we generalize the concepts of specification and topologival entropy to NAIFSs and investigate the relation between the specification property, topological entropy and topological chaocity of NAIFSs. Furthermore, a class of examples of NAIFSs is given where the
specification property holds.

The notion of topological pressure, using separated sets, was brought to the theory of dynamical systems by Ruelle \cite{DR}, later other definitions of topological pressure, based on open covers and spanning sets, were given by Walters \cite{PW} and it was further developed by Pesin and Pitskel \cite{YPBP}. Pesin \cite{YP1} used the dimension approach to the notion of topological pressure, which is based on the Caratheodory structure. Recently, there were several attempts to find suitable generalizations for other systems, see, for instance, \cite{HWZ} for non-autonomous discrete dynamical systems and \cite{DMMW,FRPV} for semigroup actions.

It is well-known that the topological pressure can be computed as the limiting complexity of the dynamical
system as the size scale approaches zero. Thus, several authors provided conditions so that the topological pressure of a dynamical system can be computed as a limit at a definite size scale. For instance, Rodrigues and Varandas \cite{FRPV} showed that the topological pressure of any continuous potential that satisfies the
bounded distortion condition can be computed as a limit at a definite size scale for any finitely generated
continuous semigroup action on a compact metric space with some kind of expansive property. Also, this
result extended to non-autonomous discrete dynamical systems by Nazarian Sarkooh and Ghane \cite{JNFG}.
In addition to generalizing the concept of topologival pressure to NAIFSs, one of the central objective of this paper is to extend this result to NAIFSs.\\ \\
\textbf{The paper is organized as follows.}
In Section \ref{section2}, we give the precise definition of an NAIFS and present an overview of the main concepts and introduce notations that will study throughout this paper. We define and study the topological entropy for NAIFSs in Section \ref{section3}. Especially, we generalize for the
case of NAIFSs the classical Bowen’s result \cite{RB2} saying that the topological entropy is concentrated on the
set of nonwandering points. Then, in Section \ref{section4}, we generalize the concept of specification to NAIFSs and characterize the entropy points for NAIFSs with the specification property and show that any NAIFS of surjective maps with the specification property has positive topological entropy and all points are entropy point.
In particular, each NAIFS with the specification property is topologically chaotic. In Section \ref{section5} we define and study the topological pressure for NAIFSs. Also, we introduce the notion of $\ast$-expansive NAIFS
and show that the topological pressure of any continuous potential can be computed as a limit at a definite size scale for every NAIFS with the $\ast$-expansive property.
Finally, in Section \ref{section6}, a special class of NAIFSs with the specification and $\ast$-expansive properties is introduced. Moreover, we illustrate two examples of NAIFSs which fit in our situation.
\section{Preliminaries}\label{section2}
Following \cite{LGMU}, a \emph{non}-\emph{autonomous iterated function system} (or NAIFS for short)
is a pair $(X,\Phi)$ in which $X$ is a set and $\Phi$ consists of a sequence $\{\Phi^{(j)}\}_{j\geq 1}$ of
collections of maps, where $\Phi^{(j)}=\{\varphi_{i}^{(j)}:X\to X\}_{i\in I^{(j)}} $ and $ I^{(j)} $ is a
non-empty finite index set for all $j\geq 1$. By $(X,\Phi_{k})$, we denote the pair of $X$ and shifted sequence
$\{\Phi^{(j)}\}_{j\geq k}$ and we use analogous notation for other sequences of objects related to an NAIFS.
If the set $X$ is a compact topological space and all the $\varphi_{i}^{(j)}$ are continuous, we speak of a \emph{topological} NAIFS. Note that in the case where all $\varphi_{i}^{(j)}$ are contraction affine similarities, this is
also referred to as a Moran set construction. For simplicity, we define the following symbolic spaces for positive integers $m,n\geq 1$:
$$ I^{m,n}:=\prod_{j=0}^{n-1}I^{(m+j)},\ \ \
I^{m,\infty}:=\prod_{j=m}^{\infty}I^{(j)}.$$
Elements of $ I^{1,n} $ are called \emph{initial} $n$-\emph{words}, while those of $I^{m,n}$
with $m>1$ are called \emph{non-initial} $n$-\emph{words}. If there is no confusion, we use
the term $n$-\emph{words} for these two cases without further characterization.

A word $w$ is called \emph{finite} if $w\in I^{m,n}$ for some $m,n\geq 1$, in this case its
\emph{length} is $n$ and denoted by $|w|:=n$. While, each word $w\in I^{m,\infty}$
is called an \emph{infinite} word and its length is infinity and denoted by $|w|:=\infty$. For finite (infinite)
word $w=w_{m}w_{m+1}\ldots w_{m+n-1}(w=w_{m}w_{m+1}\ldots)\in I^{m,n} (I^{m,\infty})$
and $1\leq k\leq |w|(1\leq k<\infty)$ we define $w|_{k}=w_{m}w_{m+1}\ldots w_{m+k-1}$ and
$w|^{k}=w_{m+k}\ldots w_{m+n-1} (w|_{k}=w_{m}w_{m+1}\ldots w_{m+k-1}\ \text{and}\
w|^{k}=w_{m+k}w_{m+k+1}\ldots)$.

The time evolution of the system is defined by composing the maps $\varphi_{i}^{(j)}$ in the obvious way. In general, for finite (infinite) word
$w=w_{m}w_{m+1}\ldots w_{m+n-1}(w=w_{m}w_{m+1}\ldots)\in I^{m,n}(I^{m,\infty})$
and $1\leq k\leq |w|(1\leq k<\infty)$ we define
\begin{equation*}\label{eq23}
\varphi_{w}^{m,k}:=\varphi_{w_{m+k-1}}^{(m+k-1)}\circ
\cdots\circ\varphi_{w_{m+1}}^{(m+1)}\circ\varphi_{w_{m}}^{(m)}\ \
\text{and}\ \ \varphi_{w}^{m,0}:=id_{X}.
\end{equation*}

We put $\varphi_{w}^{m,-k}:=(\varphi_{w}^{m,k})^{-1}$, which will be applied to sets, because we do not assume that the maps $\varphi_{i}^{(j)}$ are invertible. The \emph{orbit (trajectory)} of a point $x\in X$ is the
set $\{\varphi_{w}^{1,k}(x): k\geq 0\ \text{and}\ w\in I^{1,\infty}\}$. Also, for $w\in I^{1,\infty}$, the $w$-\emph{orbit} of $x\in X$ is the sequence $\{\varphi_{w}^{1,k}(x)\}_{k\geq 0}$.

Let NAIFS $(X,\Phi)$ and $n\geq 1$ be given. Denote by $(X,\Phi^{n})$ the NAIFS defined by
the sequence $\{\Phi^{(j,n)}\}_{j\geq 1}$, where $\Phi^{(j,n)}$ is the collection
$\{\varphi_{w_{j}^{\ast}}^{(j,n)}\}_{w_{j}^{\ast}\in I^{(j,n)}}$,
$I^{(j,n)}:=\{w_{j}^{\ast}\in I^{(j-1)n+1,n}\}$ (note that $I^{(j,n)}=I^{(j-1)n+1,n}$) and
$\varphi_{w_{j}^{\ast}}^{(j,n)}:=\varphi_{w_{jn}}^{(jn)}\circ
\cdots\circ\varphi_{w_{(j-1)n+2}}^{((j-1)n+2)}\circ\varphi_{w_{(j-1)n+1}}^{((j-1)n+1)}$
for $w_{j}^{\ast}=w_{(j-1)n+1}w_{(j-1)n+2}\ldots w_{jn}$.
Take $I_{\ast}^{m,k}:=\prod_{j=0}^{k-1}I^{(m+j,n)}$, then $\#(I_{\ast}^{1,m})=\#(I^{1,mn})$,
where $\#(A)$ is the cardinal number of the set $A$.
For $w=w_{1}w_{2}\ldots w_{mn}\in I^{1,mn}$ and $1\leq j\leq m$,
denote $w_{(j-1)n+1}w_{(j-1)n+2}\ldots w_{jn}$ by $w_{j}^{\ast}\in I^{(j,n)}$, then $w=w_{1}^{\ast}w_{2}^{\ast}\ldots w_{m}^{\ast}\in I_{\ast}^{1,m}$. For simplicity, we denote
elements in $I_{\ast}^{1,m}$ by $w^{\ast}$ and use analogous notation for other sequences of
objects related to an NAIFS.

Throughout this paper we consider topological NAIFSs $(X,\Phi)$ (except for Section \ref {section6}) so that $X$ is a compact metric space and $\Phi$ consists of a sequence $\{\Phi^{(j)}\}_{j\geq 1}$ of non-empty finite collections $\Phi^{(j)}$ of continuous self-maps.
\section{Topological entropy}\label{section3}
In this section we deal with the topological entropy of NAIFSs. First, we extend the classical definition of
toplogical entropy to NAIFSs via open covers. Then we give the Bowen-like definitions of topological entropy for NAIFSs and show that these different definitions coincide. We will also establish some basic properties for topological entropy of NAIFSs. Especially, we recover the classical Bowen's result to NAIFSs ensures that the topological entropy is concentrated on the set of nonwandering points.
\subsection{Topological entropy of NAIFSs via open covers}
In this subsection we are going to extend the definition of topological entropy to NAIFSs via open covers,
which is a natural generalization of the definition of topological entropy for autonomous dynamical systems \cite{PW}, non-autonomous discrete dynamical systems \cite{SKLS} and semigroup actions \cite{JTBLWC}.
In fact, if $\#(I^{(j)})=1$ and $\Phi^{(j)}=\{\varphi_{1}^{(j)}\}$ for every $j\geq 1$, then we get the
defnition of topological entropy for non-autonomous discrete dynamical system $(X,\varphi_{1,\infty})$,
where $\varphi_{1,\infty}$ is the sequence $\{\varphi_{1}^{(j)}\}_{j=1}^{\infty}$. Additionally, if $\varphi_{1}^{(j)}=\varphi$
for every $j\geq 1$, then we get the classical defnition of topological entropy for autonomous dynamical
system $(X,\varphi)$. Moreover, in the case that $\Phi^{(i)}=\Phi^{(j)}$ for all $i,j\geq 1$, then we get the defnition of topological entropy for semigroup action $(X,G)$ with generator set $\{\varphi_{i}^{(1)}:i\in I^{(1)}\}$.

Let $(X,\Phi)$ be an NAIFS of continuous maps on a compact topological space $X$. We define its topological entropy as follows. A family $\mathcal{A}$ of subsets of $X$ is called a \emph{cover} (of $X$) if their union is
all of $X$. For open covers $\mathcal{A}_{1},\mathcal{A}_{2},\ldots,\mathcal{A}_{n}$ of $X$ we denote
\begin{equation*}
\bigvee_{i=1}^{n}\mathcal{A}_{i}=\mathcal{A}_{1}\vee\mathcal{A}_{2}\vee
\cdots\vee\mathcal{A}_{n}=\{A_{1}\cap A_{2}\cap\cdots\cap A_{n}:
A_{i}\in\mathcal{A}_{i}\ \text{for}\ 1\leq i\leq n\}.
\end{equation*}
Note that $\bigvee_{i=1}^{n}\mathcal{A}_{i}$ is also an open cover of $X$. For an open cover
$\mathcal{A}$, finite word $w=w_{m}w_{m+1}\ldots w_{m+n-1}\in I^{m,n}$ and $0\leq j\leq n$ we
denote $\varphi_{w}^{m,-j}(\mathcal{A})=\{\varphi_{w}^{m,-j}(A): A\in\mathcal{A}\}$ and $\mathcal{A}_{w}^{m,n}:=\bigvee_{j=0}^{n}\varphi_{w}^{m,-j}(\mathcal{A})$. For each
$0\leq j\leq n$, $\varphi_{w}^{m,-j}(\mathcal{A})$ is an open cover, so $\mathcal{A}_{w}^{m,n}$
is also an open cover. Next, we denote by $\mathcal{N}(\mathcal{A})$ the \emph{minimal possible
cardinality of a subcover chosen from} $\mathcal{A}$. Then
\begin{equation*}
 h(X,\Phi;\mathcal{A}):=\limsup_{n\to\infty}\frac{1}{n}\log\Bigg(\dfrac{1}{\#(I^{1,n})}\sum_{w\in I^{1,n}}\mathcal{N}(\mathcal{A}_{w}^{1,n})\Bigg)
\end{equation*}
is said to be the \emph{topological entropy} of NAIFS $(X,\Phi)$ on the cover $\mathcal{A}$,
where $\#(I^{1,n})$ is the cardinality of the set $I^{1,n}$. The \emph{topological entropy} of
NAIFS $(X,\Phi)$ is defined by
\begin{equation*}
 h_{\text{top}}(X,\Phi):=\sup\{h(X,\Phi;\mathcal{A}): \mathcal{A}\ \text{is an open cover of}\ X\}.
\end{equation*}

For open covers $\mathcal{A}, \mathcal{B}$ of $X$, continuous map $g:X\to X$ and finite word $w\in I^{m,n}$, the following inequalities hold:
\begin{equation}\label{eq7}
\mathcal{N}(\mathcal{A}\vee\mathcal{B})\leq\mathcal{N}(\mathcal{A})\ .\ \mathcal{N}(\mathcal{B}),
\end{equation}
\begin{equation}\label{eq6}
\mathcal{N}(\varphi_{w}^{m,-n}(\mathcal{A}))\leq\mathcal{N}(\mathcal{A}),
\end{equation}
\begin{equation}\label{eq5}
g^{-1}(\mathcal{A}\vee\mathcal{B})=g^{-1}(\mathcal{A})\vee g^{-1}(\mathcal{B}).
\end{equation}

We say that a cover $\mathcal{A}$ is finer than a cover $\mathcal{B}$, and write $\mathcal{A}>\mathcal{B}$, when each element of $\mathcal{A}$ is contained in some element of $\mathcal{B}$. If $\mathcal{A}>\mathcal{B}$, then $\mathcal{N}(\mathcal{A})\geq\mathcal{N}(\mathcal{B})$ and $\mathcal{A}_{w}^{1,n}>\mathcal{B}_{w}^{1,n}$ for each $w\in I^{1,n}$. Hence,
\begin{equation}\label{n55}
\text{if}\ \mathcal{A}>\mathcal{B}\ \text{then}\ h(X,\Phi;\mathcal{A})\geq h(X,\Phi;\mathcal{B}).
\end{equation}

Since $X$ is compact, in the definition of $h_{\text{top}}(X,\Phi)$ it is sufficient to take the supremum
only over all open finite covers. If $\mathcal{A}$ is an open finite cover of $X$ and $w\in I^{1,n}$ then
the cardinality of $\mathcal{A}_{w}^{1,n}$ is at most $(\#(\mathcal{A}))^{n}$.
Therefore, $h(X,\Phi;\mathcal{A})\leq\log(\#(\mathcal{A}))$ and
so $0\leq h(X,\Phi;\mathcal{A})<\infty$. But, it can be $h_{\text{top}}(X,\Phi)=\infty$.

Now, we extend the definition of topological entropy of an NAIFS to not necessarily compact and not necessarily invariant subsets of a compact topological space. Note that the idea of defining the topological entropy for
non-compact and non-invariant sets is not new. See \cite{RB1} and \cite{YP}, where Bowen and Pesin introduce the dimension definition of topological entropy for autonomous dynamical systems, that applied to not necessarily compact and not necessarily invariant subsets of a topological space. Let $(X,\Phi)$ be an NAIFS of continuous maps on a compact topological space $X$ and $Y$ be a non-empty subset of $X$. The set $Y$ may not be
compact and may not exhibit any kind of invariance with respect to $\Phi$. If $\mathcal{A}$ is a cover of $X$ we denote by $\mathcal{A}|_{Y}$ the cover $\{A\cap Y:A\in\mathcal{A}\}$ of the set $Y$. Then we define the \emph{topological entropy} of NAIFS $(X,\Phi)$ on the set $Y$ by
\begin{equation*}
 h_{\text{top}}(Y, \Phi):=\sup\{h(Y,\Phi;\mathcal{A}): \mathcal{A}\ \text{is an open cover of}\ X\},
\end{equation*}
where
\begin{equation*}
 h(Y,\Phi;\mathcal{A}):=\limsup_{n\to\infty}\frac{1}{n}\log\Bigg(\dfrac{1}{\#(I^{1,n})}\sum_{w\in I^{1,n}}\mathcal{N}(\mathcal{A}_{w}^{1,n}|_{Y})\Bigg).
\end{equation*}
\subsection{Equivalent Bowen-like definitions of topological entropy}
Let $(X,\Phi)$ be an NAIFS of continuous maps on a compact metric space $(X,d)$. For finite (infinite)
word $w=w_{m}w_{m+1}\ldots w_{m+n-1}(w=w_{m}w_{m+1}\ldots)\in I^{m,n} (I^{m,\infty})$
and $1\leq k\leq |w|(1\leq k<\infty)$ we introduce on $X$ the \emph{Bowen}-\emph{metrics}
\begin{equation}\label{eq18}
d_{w,k}(x,y):=\max_{0\leq j\leq k}d(\varphi_{w}^{m,j}(x),\varphi_{w}^{m,j}(y)).
\end{equation}
Also, for finite (infinite)
word $w=w_{m}w_{m+1}\ldots w_{m+n-1}(w=w_{m}w_{m+1}\ldots)\in I^{m,n} (I^{m,\infty})$,
$1\leq k\leq |w|(1\leq k<\infty)$ , $x\in X$ and $\epsilon>0$, we define
\begin{equation}\label{eq14}
B(x;w,k,\epsilon):=\{y\in X: d_{w,k}(x,y)<\epsilon\},
\end{equation}
which is called the \emph{dynamical} $(k+1)$-\emph{ball} with radius $\epsilon$ relative to word $w$
around $x$.

Fix $w\in I^{1,n}$ for some $n\geq 1$. A subset $E$ of the space $X$ is called
$(n,w,\epsilon; \Phi)$-separated, if for any two distinct points $x,y\in E$, $d_{w,n}(x,y)>\epsilon$
(note that $|w|=n$). Also, a subset $F$ of the space $X$, $(n,w,\epsilon; \Phi)$-spans another subset
$K\subseteq X$, if for each $x\in K$ there is a $y\in F$ such that $d_{w,n}(x,y)\leq\epsilon$.
For subset $Y$ of $X$ we define $s_{n}(Y;w,\epsilon,\Phi)$, as the maximal cardinality of
an $(n,w,\epsilon; \Phi)$-separated set in $Y$ and $r_{n}(Y;w,\epsilon,\Phi)$ as the minimal cardinality
of a set in $Y$ which $(n,w,\epsilon; \Phi)$-spans $Y$. If $Y=X$ we sometime suppress $Y$ and shortly
write $s_{n}(w,\epsilon,\Phi)$ and $r_{n}(w,\epsilon,\Phi)$.
\begin{lemma}\label{lemma1}
Let $(X,\Phi)$ be an NAIFS of continuous maps on a compact metric space $(X,d)$ and $Y$ be a non-empty
subset of $X$. Then,
\begin{equation*}
 h_{\text{top}}(Y, \Phi)=\lim_{\epsilon\to 0}\limsup_{n\to\infty}\frac{1}{n}\log S_{n}(Y;\epsilon,\Phi)
=\lim_{\epsilon\to 0}\limsup_{n\to\infty}\frac{1}{n}\log R_{n}(Y;\epsilon,\Phi),\ \text{where}
\end{equation*}
\begin{equation*}
 S_{n}(Y;\epsilon,\Phi):=\dfrac{1}{\#(I^{1,n})}\sum_{w\in I^{1,n}}s_{n}(Y;w,\epsilon,\Phi)\ \text{and}\
 R_{n}(Y;\epsilon,\Phi):=\dfrac{1}{\#(I^{1,n})}\sum_{w\in I^{1,n}}r_{n}(Y;w,\epsilon,\Phi).
\end{equation*}
\end{lemma}
\begin{proof}
First we prove the second equality that is an immediate consequence of the following relation
\begin{equation*}
R_{n}(Y;\epsilon,\Phi)\leq S_{n}(Y;\epsilon,\Phi)\leq R_{n}(Y;\frac{\epsilon}{2},\Phi)\ \text{for all}\ \epsilon>0.
\end{equation*}
To prove this relation, it is enough to show that
\begin{equation}\label{eq2}
r_{n}(Y;w,\epsilon,\Phi)\leq s_{n}(Y;w,\epsilon,\Phi)\leq r_{n}(Y;w,\frac{\epsilon}{2},\Phi)\ \text{for all}\ \epsilon>0\ \text{and}\ w\in I^{1,n}.
\end{equation}
Fix $\epsilon>0$ and $w\in I^{1,n}$. It is obvious that any maximal $(n,w,\epsilon; \Phi)$-separated subset of $Y$ is an $(n,w,\epsilon; \Phi)$-spanning set for $Y$. Therefore $r_{n}(Y;w,\epsilon,\Phi)\leq s_{n}(Y;w,\epsilon,\Phi)$. To show the other inequality of (\ref{eq2}) suppose $E$ is an $(n,w,\epsilon; \Phi)$-separated subset of $Y$ and $F\subset X$ is an $(n,w,\frac{\epsilon}{2}; \Phi)$-spanning set of $Y$. Define $\psi:E\to F$ by choosing, for cach $x\in E$, some point $\psi(x)\in F$ with $d_{w,n}(x,\psi(x))\leq\frac{\epsilon}{2}$. Then $\psi$ is injective and therefore the cardinality of $E$ is not greater than that of $F$. Hence, $s_{n}(Y;w,\epsilon,\Phi)\leq r_{n}(Y;w,\frac{\epsilon}{2},\Phi)$. This completes the proof of relation (\ref{eq2}).

To prove the first equality, let $\epsilon>0$ and $w\in I^{1,n}$ be given. Let $E$ be an $(n,w,\epsilon; \Phi)$-separated subset of $Y$ and $\mathcal{A}$ be an open cover of $X$ by sets of diameter less than $\epsilon$. Then by definition of $(n,w,\epsilon; \Phi)$-separated sets two distinct point of $E$ cannot lie in the same element of
$\mathcal{A}\vee\varphi_{w}^{1,-1}(\mathcal{A})\vee\varphi_{w}^{1,-2}(\mathcal{A})\vee\cdots\vee\varphi_{w}^{1,-n}(\mathcal{A})$.
Therefore $s_{n}(Y;w,\epsilon,\Phi)\leq\mathcal{N}(\mathcal{A}_{w}^{1,n}|_{Y})$. Hence, by the definition of topological entropy, it follows that
\begin{equation}\label{eq3}
h_{\text{top}}(Y, \Phi)\geq\lim_{\epsilon\to 0}\limsup_{n\to\infty}\frac{1}{n}\log S_{n}(Y;\epsilon,\Phi).
\end{equation}

To prove the inverse of relation (\ref{eq3}), let $\mathcal{A}$ be an open cover of $X$ and $\lambda>0$ be a Lebesgue number for $\mathcal{A}$. Then, for every $x\in X$ and $\epsilon<\frac{\lambda}{2}$, the closed $\epsilon$-ball $B_{\epsilon}(x)$ lies inside some element $A_{\alpha}\in\mathcal{A}$. Fix $w\in I^{1,n}$. Let $F$ be an $(n,w,\epsilon; \Phi)$-spanning set of $Y$ with minimal cardinality $r_{n}(Y;w,\epsilon,\Phi)$. For each $z\in F$ and each $0\leq k\leq n$ (note that $|w|=n$), let $A_{k}(z)$ be some element of $\mathcal{A}$ containing $B_{\epsilon}(\varphi_{w}^{1,k}(z))$. On the other hand, as $F$ is an $(n,w,\epsilon; \Phi)$-spanning set of $Y$,
for any $y\in Y$ there is a $z\in F$ such that $\varphi_{w}^{1,k}(y)\in B_{\epsilon}(\varphi_{w}^{1,k}(z))$ for $0\leq k\leq n$. Thus, $\varphi_{w}^{1,k}(y)\in A_{k}(z)$ for $0\leq k\leq n$, and the family
\begin{equation*}
 \big\{A_{0}(z)\cap\varphi_{w}^{1,-1}(A_{1}(z))\cap\cdots\cap\varphi_{w}^{1,-n}(A_{n}(z))\cap Y: z\in F\big\}
\end{equation*}
is a subcover of the cover $\mathcal{A}_{w}^{1,n}|_{Y}$ of $Y$. Hence, $\mathcal{N}(\mathcal{A}_{w}^{1,n}|_{Y})\leq \#(F)=r_{n}(Y;w,\epsilon,\Phi)$. Now, by the definition of topological entropy and second equality, we get
\begin{equation*}
h_{\text{top}}(Y, \Phi)\leq\lim_{\epsilon\to 0}\limsup_{n\to\infty}\frac{1}{n}\log R_{n}(Y;\epsilon,\Phi)=\lim_{\epsilon\to 0}\limsup_{n\to\infty}\frac{1}{n}\log S_{n}(Y;\epsilon,\Phi),
\end{equation*}
which completes the proof.
\end{proof}
\begin{remark}
The following two facts hold:
\begin{itemize}
\item The limits in the previous lemma can be replaced by $\sup_{\epsilon>0}$, because for $\epsilon_{2}<\epsilon_{1}$ and $w\in I^{1,n}$ we have
$$r_{n}(Y;w,\epsilon_{2},\Phi)\geq r_{n}(Y;w,\epsilon_{1},\Phi)\ \ \text{and}\ \
s_{n}(Y;w,\epsilon_{2},\Phi)\geq s_{n}(Y;w,\epsilon_{1},\Phi).$$
\item $r_{n}(Y;w,\epsilon,\Phi)$ is defined for $w\in I^{1,n}$ as the minimal cardinality of a
set in $Y$ which $(n,w,\epsilon; \Phi)$-spans $Y$. If we take $r_{n}^{X}(Y;w,\epsilon,\Phi)$ for
$w\in I^{1,n}$ as the minimal cardinality of a set in $X$ which $(n,w,\epsilon; \Phi)$-spans $Y$,
again we have
\begin{equation*}
 h_{\text{top}}(Y, \Phi)=\lim_{\epsilon\to 0}\limsup_{n\to\infty}\frac{1}{n}\log R^{X}_{n}(Y;\epsilon,\Phi),
\end{equation*}
where
\begin{equation*}
 R_{n}^{X}(Y;\epsilon,\Phi):=\dfrac{1}{\#(I^{1,n})}\sum_{w\in I^{1,n}}r_{n}^{X}(Y;w,\epsilon,\Phi).
\end{equation*}
Hence, it is not important that we take $r_{n}(Y;w,\epsilon,\Phi)$ for $w\in I^{1,n}$ as the minimal cardinality of a set in $Y$ which $(n,w,\epsilon;\Phi)$-spans $Y$ or as the minimal cardinality of a set in $X$
which $(n,w,\epsilon;\Phi)$-spans $Y$.
\end{itemize}
\end{remark}
\subsection{Basic properties of topological entropy}
In this subsection we are going to give the basic properties of topological entropy of NAIFSs.
\begin{lemma}\label{lemma2}
Let for $1\leq i\leq k$, $n=1,2,\ldots$ and $w\in I^{n}$ in which $I^{n}$ is a non-empty finite
set, $a_{n,w,i}$'s be non-negative numbers. Then
\begin{equation*}
 \limsup_{n\to\infty}\frac{1}{n}\log\Big(\frac{1}{\#(I^{n})}\sum_{\substack{w\in I^{n} \\
1\leq i\leq k}}a_{n,w,i}\Big)=\max_{1\leq i\leq k}\limsup_{n\to\infty}\frac{1}{n}\log\Big(\frac{1}{\#(I^{n})}\sum_{w\in I^{n}}a_{n,w,i}\Big).
\end{equation*}
\end{lemma}
\begin{proof}
It is actually a direct consequence of a priori simpler expression considered for non-autonomous dynamical systems (see \cite[Lemma 4.1]{SKLS}
and \cite[Lemma 4.1.9]{LAJLMM}),
taking
$$a_{n,i}:=\frac{1}{\#(I^{n})}\sum_{w\in I^{n}}a_{n,w,i}.$$
\end{proof}
\begin{proposition}\label{proposition1}
Let $(X,\Phi)$ be an NAIFS of continuous maps on a compact topological space $X$. If $X=\bigcup_{i=1}^{k}X_{i}$ in which each $X_{i}$ is an arbitrary non-empty subset of $X$, then
\begin{equation*}
 h_{\text{top}}(X,\Phi)=\max_{1\leq i\leq k}h_{\text{top}}(X_{i},\Phi).
\end{equation*}
\end{proposition}
Note that, we do not need to assume that the sets $X_{i}$ are closed or invariant (invariant in the sense that
they contain the trajectories of all points), because we have defined the topological entropy of NAIFS $(X,\Phi)$ on every subset of $X$.
\begin{proof}
By the definition of topological entropy we have $h_{\text{top}}(X,\Phi)\geq\max_{1\leq i\leq k}h_{\text{top}}(X_{i},\Phi)$. To prove the reverse inequality, let $w\in I^{1,n}$ and $\mathcal{A}$ be an open cover of $X$. Let $\mathcal{C}_{1}, \mathcal{C}_{2}, \ldots, \mathcal{C}_{k}$ be subcovers chosen from the covers $\mathcal{A}_{w}^{1,n}|_{X_{1}},\mathcal{A}_{w}^{1,n}|_{X_{2}}, \ldots, \mathcal{A}_{w}^{1,n}|_{X_{k}}$, respectively. Then each element of $\mathcal{C}=\mathcal{C}_{1}\cup\mathcal{C}_{2}\cup\cdots\cup\mathcal{C}_{k}$ is contained in some element of $\mathcal{A}_{w}^{1,n}$ and $\mathcal{C}$ is an open cover of $X$. This implies
$\mathcal{N}(\mathcal{A}_{w}^{1,n})\leq\sum_{i=1}^{k}\mathcal{N}(\mathcal{A}_{w}^{1,n}|_{X_{i}})$.
Now, by Lemma \ref{lemma2}, we get
\begin{eqnarray*}
h(X,\Phi;\mathcal{A})
&=& \limsup_{n\to\infty}\frac{1}{n}\log\Big(\dfrac{1}{\#(I^{1,n})}\sum_{w\in I^{1,n}}\mathcal{N}(\mathcal{A}_{w}^{1,n})\Big)\\
&\leq & \limsup_{n\to\infty}\frac{1}{n}\log\Big(\dfrac{1}{\#(I^{1,n})}\sum_{\substack{w\in I^{1,n} \\
1\leq i\leq k}}\mathcal{N}(\mathcal{A}_{w}^{1,n}|_{X_{i}})\Big)\\
&=& \max_{1\leq i\leq k}\limsup_{n\to\infty}\frac{1}{n}\log\Big(\dfrac{1}{\#(I^{1,n})}
\sum_{w\in I^{1,n}}\mathcal{N}(\mathcal{A}_{w}^{1,n}|_{X_{i}})\Big)\\
&=& \max_{1\leq i\leq k}h(X_{i},\Phi;\mathcal{A})\leq\max_{1\leq i\leq k}h_{\text{top}}(X_{i},\Phi).
\end{eqnarray*}
Since open cover $\mathcal{A}$ was arbitrary, we conclude that $h_{\text{top}}(X,\Phi)\leq\max_{1\leq i\leq k}h_{\text{top}}(X_{i},\Phi)$, which completes the proof.
\end{proof}
Now, we give an analogue of the well known property
$h_{\text{top}}(\varphi^{n})=n\ .\ h_{\text{top}}(\varphi)$ of the topological entropy of autonomous
dynamical systems to NAIFSs that will be used in the proof of Theorem \ref{theorem10}.

The following result is now folklore and we omit its
proof, see \cite[Lemma 4.2]{SKLS}.
\begin{lemma}\label{lemma3}
Let $(X,\Phi)$ be an NAIFS of continuous maps on a compact topological space $X$. Then for any subset $Y$ of $X$ and every $n\geq 1$,
$h_{\text{top}}(Y,\Phi^{n})\leq n\ .\ h_{\text{top}}(Y,\Phi)$.
\end{lemma}
\begin{remark}\label{remark10}
In general, we cannot claim that $h_{\text{top}}(X,\Phi^{n})=n\ .\ h_{\text{top}}(X,\Phi)$(see the comment after Lemma 4.2 in \cite{SKLS}, where $\#(I^{(j)})=1$ for every $j\in\mathbb{N}$). Note that the results in \cite{SKLS}
are about non-autonomous discrete dynamical systems which are a special case of NAIFSs.
\end{remark}
Now, we give some sufficient conditions to have equality in Lemma \ref{lemma3}. An NAIFS $(X,\Phi)$ of continuous maps on a compact metric space $(X,d)$ is said to be
\emph{equicontinuous}, if for every $\epsilon>0$ there exists $\delta>0$ such that the implication $d(x,y)<\delta\Rightarrow d(\varphi_{i}^{(j)}(x),\varphi_{i}^{(j)}(y))<\epsilon$ holds for
every $x,y\in X$, $j\geq 1$ and $i\in I^{(j)}$.

By Lemma \cite[Lemma 4.4]{SKLS} the following results can be followed.
\begin{lemma}\label{theorem4}
Let $(X,\Phi)$ be an equicontinuous NAIFS on a compact metric space $(X,d)$.
Then for any subset $Y$ of $X$ and every $n\geq 1$, $h_{\text{top}}(Y,\Phi^{n})= n\ .\ h_{\text{top}}(Y,\Phi)$.
\end{lemma}
Let us take an NAIFS $(X,\Phi)$ in which $X$ is a compact metric space and $\Phi$ consists of a sequence $\{\Phi^{(j)}\}_{j\geq 1}$ of
collections of maps, where $\Phi^{(j)}=\{\varphi_{i}^{(j)}:X\to X\}_{i\in I^{(j)}} $ and $ I^{(j)} $ is a
non-empty finite index set for all $j\geq 1$. For each $k \geq 1$ we will denote by $(X,\Phi_{k})$ the NAIFS composed of the sequence $\{\Phi^{(j)}\}_{j\geq k}$.

Now, we give the following lemma that will be used in the next section.
\begin{lemma}\label{theorem1}
Let $(X, \Phi)$ be an NAIFS of continuous maps on a compact topological space $X$.
Then $h(X,\Phi_{i};\mathcal{A})\leq h(X,\Phi_{j};\mathcal{A})$ for every $1\leq i\leq j<\infty$ and every
open cover $\mathcal{A}$ of $X$.
In particular, $h_{\text{top}}(X,\Phi_{i})\leq h_{\text{top}}(X,\Phi_{j})$.
\end{lemma}
\begin{proof}
It is enough to show that $h(X,\Phi_{i};\mathcal{A})\leq h(X,\Phi_{i+1};\mathcal{A})$, for
every $1\leq i<\infty$ and every open cover $\mathcal{A}$ of $X$.
Let $i\geq 1$ and $\mathcal{A}$ be an open cover of $X$. For $w=w_{i}w_{i+1}\ldots w_{i+n-1}\in I^{i,n}$ put $w^{\prime}:=w|^{1}=w_{i+1}\ldots w_{i+n-1}\in I^{i+1,n-1}$. Now, by relation (\ref{eq5}), we have
\begin{eqnarray*}
\mathcal{A}_{w}^{i,n}
&=& \mathcal{A}\vee\varphi_{w}^{i,-1}(\mathcal{A})\vee\varphi_{w}^{i,-2}(\mathcal{A})\vee\cdots\vee\varphi_{w}^{i,-n}(\mathcal{A})\\
&=& \mathcal{A}\vee\big(\varphi_{w_{i}}^{(i)}\big)^{-1}\big(\mathcal{A}\vee
\varphi_{w^{\prime}}^{i+1,-1}(\mathcal{A})\vee\varphi_{w^{\prime}}^{i+1,-2}(\mathcal{A})\vee\cdots\vee\varphi_{w^{\prime}}^{i+1,-(n-1)}(\mathcal{A})\big)\\
&=& \mathcal{A}\vee\big(\varphi_{w_{i}}^{(i)}\big)^{-1}\big(\mathcal{A}_{w^{\prime}}^{i+1,n-1}\big).
\end{eqnarray*}
Using relations (\ref{eq7}) and (\ref{eq6}) we get
\begin{eqnarray*}
h(X,\Phi_{i};\mathcal{A})
&=& \limsup_{n\to\infty}\frac{1}{n}\log\Big(\dfrac{1}{\#(I^{i,n})}\sum_{w\in I^{i,n}}\mathcal{N}(\mathcal{A}_{w}^{i,n})\Big)\\
&\leq& \limsup_{n\to\infty}\frac{1}{n}\log\Big(\dfrac{1}{\#(I^{i,n})}\sum_{w\in I^{i,n}}\mathcal{N}(\mathcal{A})\ .\ \mathcal{N}(\mathcal{A}_{w^{\prime}}^{i+1,n-1})\Big)\\
&=& \limsup_{n\to\infty}\frac{1}{n}\log\Big(\dfrac{\#(I^{(i)})}{\#(I^{i,n})}\sum_{w^{\prime}\in I^{i+1,n-1}}\mathcal{N}(\mathcal{A})\ .\ \mathcal{N}(\mathcal{A}_{w^{\prime}}^{i+1,n-1})\Big)\\
&=& \limsup_{n\to\infty}\frac{1}{n}\log\Big(\dfrac{\mathcal{N}(\mathcal{A})}{\#(I^{i+1,n-1})}\sum_{w^{\prime}\in I^{i+1,n-1}}\mathcal{N}(\mathcal{A}_{w^{\prime}}^{i+1,n-1})\Big)\\
&=& \limsup_{n\to\infty}\frac{1}{n}\log\mathcal{N}(\mathcal{A})+\limsup_{n\to\infty}\frac{1}{n}\log\Big(\dfrac{1}{\#(I^{i+1,n-1})}\sum_{w^{\prime}\in I^{i+1,n-1}}\mathcal{N}(\mathcal{A}_{w^{\prime}}^{i+1,n-1})\Big)\\
&=&\limsup_{n\to\infty}\frac{1}{n}\log\Big(\dfrac{1}{\#(I^{i+1,n})}\sum_{w\in I^{i+1,n}}\mathcal{N}(\mathcal{A}_{w}^{i+1,n})\Big)\\
&=& h(X,\Phi_{i+1};\mathcal{A}).
\end{eqnarray*}
Now, by taking supremum over all open covers $\mathcal{A}$ of $X$ we have
$h_{\text{top}}(X,\Phi_{i})\leq h_{\text{top}}(X,\Phi_{i+1})$ which completes the proof.
\end{proof}
In general, without more assumptions, we cannot claim that $h_{\text{top}}(X,\Phi)=h_{\text{top}}(X,\Phi_{i})$
for all $i\geq 1$. Nevertheless, in Corollary \ref{corollary1} we will give a sufficient condition that guarantees the equality
$h_{\text{top}}(X,\Phi)=h_{\text{top}}(X,\Phi_{i})$ for all $i\geq 1$.
\begin{remark}\label{remark0}
Because, in general, the inequality
$\mathcal{N}\Big(\big(\varphi_{w_{i}}^{(i)}\big)^{-1}(\mathcal{A})|_{Y}\Big)
\leq\mathcal{N}(\mathcal{A}|_{Y})$ is not true, the proof of Lemma \ref{theorem1} cannot be modified
to prove an analogue of the theorem for the topological entropy on the subsets $Y$ of $X$. Hence, it is not very surprising that such an analogue does not hold (see \cite[Fig.2 and comments]{SKLS}, where $\#(I^{(j)})=1$
for every $j\in\mathbb{N}$).
\end{remark}
\subsection{Asymptotical topological entropy and topologically chaotic NAIFSs}
As an autonomous dynamical system $(X,f)$ is usually called topologically chaotic if $h_{\text{top}}(f)>0$, one could consider also an NAIFS $(X,\Phi)$ with $h_{\text{top}}(X,\Phi)>0$ to be topologically chaotic. But, we
give another definition which is an extension of the definition of topologically chaotic that given by Kolyada and Snoha for non-autonomous discrete dynamical systems \cite{SKLS}.

Let $(X,\Phi)$ be an NAIFS of continuous maps on a compact topological space $X$ and
$\mathcal{A}$ be an open cover of $X$, then by Lemma \ref{theorem1} the limit
\begin{equation*}
 h^{\ast}(X,\Phi;\mathcal{A}):=\lim_{n\to\infty}h(X,\Phi_{n};\mathcal{A})=\lim_{n\to\infty}\limsup_{k\to\infty}\frac{1}{k}\log\Bigg(\dfrac{1}{\#(I^{n,k})}\sum_{w\in I^{n,k}}\mathcal{N}(\mathcal{A}_{w}^{n,k})\Bigg)
\end{equation*}
exists. The quantity $h^{\ast}(X,\Phi;\mathcal{A})$ is said to be the \emph{asymptotical topological entropy of
the} NAIFS $(X,\Phi)$ \emph{on the cover} $\mathcal{A}$. Put
\begin{equation*}
h^{\ast}(X,\Phi):=\sup_{\mathcal{A}}h^{\ast}(X,\Phi;\mathcal{A})
\end{equation*}
where the supremum is taken over all open covers $\mathcal{A}$ of $X$. By the definition and
Lemma \ref{theorem1} it is easy to see that
\begin{eqnarray*}
h^{\ast}(X,\Phi)
&=& \sup_{\mathcal{A}}h^{\ast}(X,\Phi;\mathcal{A})=\sup_{\mathcal{A}}\lim_{n\to\infty}h(X,\Phi_{n};\mathcal{A})=\sup_{\mathcal{A}}\sup_{n}h(X,\Phi_{n};\mathcal{A})\\
&=& \sup_{(\mathcal{A},n)}h(X,\Phi_{n};\mathcal{A})=\sup_{n}\sup_{\mathcal{A}}h(X,\Phi_{n};\mathcal{A})=\lim_{n\to\infty}\sup_{\mathcal{A}}h(X,\Phi_{n};\mathcal{A})\\
&=& \lim_{n\to\infty}h_{\text{top}}(X,\Phi_{n}).
\end{eqnarray*}
If $X$ is a compact metric space, then by the definition of topological entropy via separated and
spanning sets we have
\begin{eqnarray*}
h^{\ast}(X,\Phi)
&=& \lim_{n\to\infty}\lim_{\epsilon\to 0}\limsup_{k\to\infty}\frac{1}{k}\log S_{k}(\epsilon,\Phi_{n})=\lim_{\epsilon\to 0}\lim_{n\to\infty}\limsup_{k\to\infty}\frac{1}{k}\log S_{k}(\epsilon,\Phi_{n})\\
&=&\lim_{n\to\infty}\lim_{\epsilon\to 0}\limsup_{k\to\infty}\frac{1}{k}\log R_{k}(\epsilon,\Phi_{n})=\lim_{\epsilon\to 0}\lim_{n\to\infty}\limsup_{k\to\infty}\frac{1}{k}\log R_{k}(\epsilon,\Phi_{n}),
\end{eqnarray*}
where
\begin{equation*}
 S_{k}(\epsilon,\Phi_{n})=\dfrac{1}{\#(I^{n,k})}
 \sum_{w\in I^{n,k}}s_{k}(w,\epsilon,\Phi_{n})\ \ \ \text{and}\ \ \ R_{k}(\epsilon,\Phi_{n})=\dfrac{1}{\#(I^{n,k})}\sum_{w\in I^{n,k}}r_{k}(w,\epsilon,\Phi_{n}).
\end{equation*}
The quantity $h^{\ast}(X,\Phi)$ is said to be \emph{the asymptotical topological entropy of} NAIFS $(X,\Phi)$.
\begin{definition}
An NAIFS $(X,\Phi)$ of continuous maps on a compact topological space $X$ is said to be
\emph{topologically chaotic} if it has positive asymptotical topological entropy, i.e. $h^{\ast}(X,\Phi)>0$.
\end{definition}
\begin{remark}
By Remark \ref{remark0}, since for a proper subset $Y$ of $X$ ($Y\subsetneqq X$) we may have
$h_{\text{top}}(Y,\Phi_{i})\geq h_{\text{top}}(Y,\Phi_{j})$ for some $j>i$, there is a problem with the
extension of the concept of asymptotical topological entropy to a proper subset $Y$ of $X$. But, we can
define $h^{\ast}(Y,\Phi):=\limsup_{n\to\infty}h_{\text{top}}(Y,\Phi_{n})$ for proper subsets $Y$ of $X$.
\end{remark}
Many results that hold for the topological entropy of NAIFSs can be carried to asymptotical topological entropy of NAIFSs. Hence, it is not difficult  to see that Proposition \ref{proposition1}, Lemmas \ref{lemma3},
 \ref{theorem4} and \ref{theorem1} have analogues versions for asymptotical topological
entropy of NAIFSs by replacing $h_{\text{top}}$ by $h^{\ast}$.
\subsection{Entropy of NAIFSs of monotone interval maps or circle maps}
Sometimes in computing the topological entropy of a dynamical system, one may be very interested in
whether it is positive or zero rather than its exact value. Also, computing the exact value may be impossible.
In the theory of autonomous dynamical systems, a homeomorphism on the interval or the circle has zero
topological entropy (see, e.g., \cite{AKM,PW}). Also, in \cite{SKLS} in the theory of non-autonomous discrete
dynamical systems, Kolyada and Snoha showed that any non-autonomous discrete dynamical systems of
continuous (not necessarily strictly) monotone maps on the interval or the circle, have zero topological
entropy. In the following theorem, we extend these results to NAIFSs on the interval and the circle.

We consider the unit circle $S^{1}$ as the quotient space of the real line by the group of translations by
integers ($S^{1}=\mathbb{R}/\mathbb{Z}$). Let $q:\mathbb{R}\to S^{1}$ be the quotient map. In the unit circle $S^{1}$,
we consider the metric (denoted by $\rho$) and the orientation induced from the metric and orientation
of the real line via $q$ (hence the distance between any two points is at most $\frac{1}{2}$). Also, we
denote by $I$ the unit interval $[0,1]$.
Note that a homeomorphism of $I$ or $S^{1}$ is either strictly increasing (orientation preserving) or strictly decreasing (orientation reversing). The desired result can be followed from the following theorem. In it,
when we speak about an NAIFS of monotone maps we do not assume that the type of monotonicity is the
same for all of them.
\begin{theorem}
Let $(X,\Phi)$ be an NAIFS of continuous monotone maps in which $X$ is $I$ or $S^{1}$. Then, the topological entropy $h_{\text{top}}(X,\Phi)$ is zero. Thus, $h^{\ast}(X,\Phi)=0$.
\end{theorem}
\begin{proof}
First, we begin the proof for the interval case. Fix $w\in I^{1,n}$. Let $E:=\{x_{1},x_{2},\ldots,x_{k}\}$ be a subset of $I$ with
$x_{1}<x_{2}<\cdots<x_{k}$. Since the maps $\varphi_{w_{1}},\varphi_{w_{2}},\ldots,\varphi_{w_{n}}$ are monotone, for every $0\leq j\leq n$ either
$\varphi_{w}^{1,j}(x_{1})\leq\varphi_{w}^{1,j}(x_{2})\leq\cdots\leq\varphi_{w}^{1,j}(x_{k})\ \ \text{or}\ \ \varphi_{w}^{1,j}(x_{1})\geq\varphi_{w}^{1,j}(x_{2})\geq\cdots\geq\varphi_{w}^{1,j}(x_{k}).$
This implies that the set $E$ is $(w,n,\epsilon;\Phi)$-separated if and only if for every $1\leq i\leq k-1$ the set $\{x_{i},x_{i+1}\}$ is $(w,n,\epsilon;\Phi)$-separated. Since the length of the interval $I$ is $1$, for every $0\leq j\leq n$ at most $[1/\epsilon]$ distances from $|\varphi_{w}^{1,j}(x_{1})-\varphi_{w}^{1,j}(x_{2})|$,
$|\varphi_{w}^{1,j}(x_{2})-\varphi_{w}^{1,j}(x_{3})|$,
$\ldots,|\varphi_{w}^{1,j}(x_{k-1})-\varphi_{w}^{1,j}(x_{k})|$ are longer than $\epsilon$, where $[1/\epsilon]$ is the integer part of $1/\epsilon$. Hence, at most $(n+1)[1/\epsilon]$ sets of the form $\{x_{i},x_{i+1}\}$, $1\leq i\leq k-1$ are $(w,n,\epsilon;\Phi)$-separated. So if $E$ is $(w,n,\epsilon;\Phi)$-separated then $k-1\leq(n+1)[1/\epsilon]$. Consequently, $s_{n}(w,\epsilon,\Phi)\leq 1+(n+1)[1/\epsilon]$. Hence, by the definition of topological entropy, it follows that
\begin{eqnarray*}
h_{\text{top}}(I, \Phi)
&=& \lim_{\epsilon\to 0}\limsup_{n\to\infty}\frac{1}{n}\log\Bigg(\dfrac{1}{\#(I^{1,n})}
\sum_{w\in I^{1,n}}s_{n}(w,\epsilon,\Phi)\Bigg)\\
&\leq & \lim_{\epsilon\to 0}\limsup_{n\to\infty}\frac{1}{n}\log\Bigg(\dfrac{1}{\#(I^{1,n})}
\sum_{w\in I^{1,n}}\big(1+(n+1)[1/\epsilon]\big)\Bigg)\\
&=& \lim_{\epsilon\to 0}\limsup_{n\to\infty}\frac{1}{n}\log\big(1+(n+1)[1/\epsilon]\big)=0.
\end{eqnarray*}
Now, let $X=S^{1}$ and $w\in I^{1,n}$. Let $E:=\{x_{1},x_{2},\ldots,x_{k}\}$ be a maximal
$(w,n,\epsilon;\Phi)$-separated set in $S^{1}$ with $x_{1}<x_{2}<\cdots<x_{k}$, i.e. $s_{n}(w,\epsilon,\Phi)=k$. Then  the sets $\{x_{i},x_{i+1}\}$, $1\leq i\leq k-1$ and $\{x_{k},x_{1}\}$ are $(w,n,\epsilon;\Phi)$-separated. Since for every $0\leq j\leq n$ the sum of distances
$$\sum_{i=1}^{k-1}\rho(\varphi_{w}^{1,j}(x_{i}),\varphi_{w}^{1,j}(x_{i+1}))+\rho(\varphi_{w}^{1,j}(x_{k}),\varphi_{w}^{1,j}(x_{1}))$$
equals to the lenght of the circle $=1$, at most $[1/\epsilon]$ of them are longer than $\epsilon$. Hence, at most $(n+1)[1/\epsilon]$ sets of the form $\{x_{i},x_{i+1}\}$, $1\leq i\leq k-1$ or $\{x_{k},x_{1}\}$ are $(w,n,\epsilon;\Phi)$-separated. Thus $s_{n}(w,\epsilon,\Phi)=k\leq (n+1)[1/\epsilon]$, since all of these sets are $(w,n,\epsilon;\Phi)$-separated. Hence, by the definition of topological entropy, it follows that
\begin{eqnarray*}
h_{\text{top}}(S^{1}, \Phi)
&=& \lim_{\epsilon\to 0}\limsup_{n\to\infty}\frac{1}{n}\log\Bigg(\dfrac{1}{\#(I^{1,n})}
\sum_{w\in I^{1,n}}s_{n}(w,\epsilon,\Phi)\Bigg)\\
&\leq & \lim_{\epsilon\to 0}\limsup_{n\to\infty}\frac{1}{n}\log\Bigg(\dfrac{1}{\#(I^{1,n})}
\sum_{w\in I^{1,n}}(n+1)[1/\epsilon]\Bigg)\\
&=& \lim_{\epsilon\to 0}\limsup_{n\to\infty}\frac{1}{n}\log\big((n+1)[1/\epsilon]\big)=0.
\end{eqnarray*}
In a similar way, for every $n\geq 2$ one can conclude that $h_{\text{top}}(I, \Phi_{n})=0=h_{\text{top}}(S^{1}, \Phi_{n})$. Hence, $h^{\ast}(I, \Phi)=0=h^{\ast}(S^{1}, \Phi)$ which completes the proof.
\end{proof}
\subsection{Topological entropy on the set of nonwandering points}
If $(X,\varphi)$ is an autonomous dynamical system in which $\varphi$ is a continuous self-map of a compact topological space $X$, then by \cite{RB2}, the topological entropy of $\varphi$ and of $\varphi|_{\Omega(f)}$
are equal. Where, $\Omega(\varphi)$ is the set of nonwandering points of $\varphi$. A point $x\in X$ is said to
be a nonwandering point of $\varphi$ if for every non-empty open neighborhood $U_{x}$ of $x$ in $X$, there
exists a positive integer $n$ such that $\varphi^{n}(U_{x})\cap U_{x}\neq\emptyset$.
Also, in the context of non-autonomous discrete dynamical systems, Kolyada and Snoha \cite{SKLS} showed
that for every sequence $\varphi_{1,\infty}=\{\varphi_{i}\}_{i=1}^{\infty}$ of equicontinuous self-maps
of a compact metric space $X$, the topological entropy of non-autonomous discrete dynamical
system $(X,\varphi_{1,\infty})$ is equal to the topological entropy of its restriction to the set of
nonwandering points, i.e.
$h_{\text{top}}(\varphi_{1,\infty})=h_{\text{top}}(\varphi_{1,\infty}|_{\Omega(\varphi_{1,\infty})})$.
Where, $\Omega(\varphi_{1,\infty})$ is the set of nonwandering points of sequence $\varphi_{1,\infty}$.
Additionally, Eberlein \cite{EE} asserted that the topological entropy of an (abelian) finitely generated semigroup action is equal to the topological entropy of its restriction to its nonwandering set.
In the following theorem, we want to find a analogous result for NAIFSs.
\begin{definition}\label{nonwan}
Let $(X,\Phi)$ be an NAIFS of continuous maps on a compact topological space $X$. A point $x\in X$ is said to
be \emph{nonwandering} for $\Phi$ if for every open neighbourhood $U_{x}$ of $x$ there is a finite word
$w\in I^{m,n}$ for some $m,n\geq 1$, such that $\varphi_{w}^{m,n}(U_{x})\cap U_{x}\neq\emptyset$.
The set of all nonwandering points of $\Phi$ is called the \emph{nonwandering set} of $\Phi$ and denoted by $\Omega(\Phi)$. It is easy to see that $\Omega(\Phi)$ is a closed subset of $X$.
\end{definition}
\begin{remark}
The following two facts hold:
\begin{itemize}
  \item Definition \ref{nonwan} implies that an open subset $U\subseteq X$ is \emph{wandering} for $\Phi$ if $\varphi_{w}^{m,n}(U)\cap U=\emptyset$ for every finite word $w\in I^{m,n}$
and every $m,n\geq 1$. Also a point $x\in X$ is \emph{wandering} for $\Phi$ if it belongs to some wandering set $U$. Hence, $x$ is wandering if and only if it is not nonwandering.
  \item In an NAIFS $(X,\Phi)$, if $\#(I^{(j)})=1$ and $\Phi^{(j)}=\{\varphi_{1}^{(j)}\}$ for every $j\geq 1$,
then $(X,\Phi)$ is a non-autonomous discrete dynamical system and Definition \ref{nonwan} coincides with the usual definition of nonwandering points for non-autonomous discrete dynamical systems. Additionally, if $\varphi_{1}^{(j)}=\varphi$ for every $j\geq 1$, then $(X,\Phi)$ is an autonomous dynamical system and Definition \ref{nonwan} coincides with the usual definition of nonwandering points for autonomous dynamical systems.
\end{itemize}
\end{remark}
\begin{theorem}\label{theorem10}
Let $(X,\Phi)$ be an equicontinuous NAIFS of a compact metric space $(X,d)$.
Then $h_{\text{top}}(X,\Phi)=h_{\text{top}}(\Omega(\Phi),\Phi|_{\Omega(\Phi)})$.
\end{theorem}
\begin{proof}
By the definition of topological entropy, we have
$h_{\text{top}}(X,\Phi)\geq h_{\text{top}}(\Omega(\Phi),\Phi|_{\Omega(\Phi)})$. Hence, it is enough to
prove the converse inequality. To do this we will follow the main ideas from the proof
of \cite[Theorem H]{SKLS} and \cite[Lemma 4.1.5]{LAJLMM}.

So let $\mathcal{A}$ be an open cover of $X$. Fix $n\geq 1$ and $w\in I^{1,n}$. Let $\zeta_{w}$ be an
minimal subcover of $\Omega(\Phi)$ chosen from $\mathcal{A}_{w}^{1,n}$. Since $X$ is a compact metric
space, the set $K=X\setminus \bigcup_{B\in\zeta_{w}}B$ is compact and consists of wandering points.
Hence, we can cover $K$ with a finite number of wandering sets (subsets of $X$, not necessarily of $K$),
each of them contained in some element of $\mathcal{A}_{w}^{1,n}$. These sets,
together with all elements of $\zeta_{w}$, form an open cover $\xi_{w}$ of $X$, finer than $\mathcal{A}_{w}^{1,n}$.
Now, in NAIFS $(X,\Phi^{n})$ for
$w^{\ast}=w_{1}^{\ast}w_{2}^{\ast}\ldots w_{k}^{\ast}\in I_{\ast}^{1,k}$ with $w^{\ast}=w^{\prime}=w_{1}^{\prime}w_{2}^{\prime}\ldots w_{kn}^{\prime}\in I^{1,kn}$, consider
any non-empty element of $(\xi_{w})_{w^{\ast}}^{1,k}$. It is of the form
\begin{equation*}
C_{0}\cap(\varphi_{w_{1}^{\ast}}^{(1,n)})^{-1}(C_{1})
\cap(\varphi_{w_{1}^{\ast}}^{(1,n)})^{-1}\circ(\varphi_{w_{2}^{\ast}}^{(2,n)})^{-1}(C_{2})
\cap\cdots\cap(\varphi_{w_{1}^{\ast}}^{(1,n)})^{-1}\circ
(\varphi_{w_{2}^{\ast}}^{(2,n)})^{-1}\circ
\cdots\circ(\varphi_{w_{k}^{\ast}}^{(k,n)})^{-1}(C_{k}),
\end{equation*}
that is equal to
\begin{equation*}
C_{0}\cap\varphi_{w^{\prime}}^{1,-n}(C_{1})\cap\varphi_{w^{\prime}}^{1,-n}\circ
\varphi_{w^{\prime}}^{n+1,-n}(C_{2})\cap\cdots\cap\varphi_{w^{\prime}}^{1,-n}\circ
\varphi_{w^{\prime}}^{n+1,-n}
\circ\cdots\circ\varphi_{w^{\prime}}^{(k-1)n+1,-n}(C_{k}),
\end{equation*}
where $\varphi_{w_{j}^{\ast}}^{(j,n)}=\varphi_{w^{\prime}}^{(j-1)n+1,n}\in \Phi^{(j,n)}$ for
$1\leq j\leq k$ and $C_{i}\in\xi_{w}$ for $0\leq i\leq k$. Since we assume that this element is non-empety,
we get that if $C_{i}=C_{_{j}}$ for some $i<j$, then
\begin{equation*}
 \varphi_{w^{\prime}}^{1,-n}\circ\cdots\circ\varphi_{w^{\prime}}^{(i-1)n+1,-n}\circ\varphi_{w^{\prime}}^{in+1,-n}
\circ\cdots\circ\varphi_{w^{\prime}}^{(j-1)n+1,-n}(C_{i})\cap\varphi_{w^{\prime}}^{1,-n}
\circ\cdots\circ\varphi_{w^{\prime}}^{(i-1)n+1,-n}(C_{i})\neq\emptyset,
\end{equation*}
so $(\varphi_{w^{\prime}}^{(j-1)n+1,n}\circ\cdots\circ\varphi_{w^{\prime}}^{in+1,n})(C_{i})=\varphi_{w^{\prime}}^{in+1,(j-i)n}(C_{i})$ intersects $C_{i}$, hence $C_{i}$ cannot be wandering
for $\Phi$, this implies that $C_{i}\in\zeta_{w}$.

One can show that \cite[Lemma 4.1.5]{LAJLMM} the number
of elements in cover $(\xi_{w})_{w^{\ast}}^{1,k}$ is not larger than
$(m+1)!\ .\ (k+1)^{m}\ .\ (\#(\zeta_{w}))^{k+1}$, where $m=\#(\xi_{w}\setminus\zeta_{w})$.
Thus,
\begin{eqnarray*}
h(X,\Phi^{n};\xi_{w})
&=& \limsup_{k\to\infty}\frac{1}{k}\log\Bigg(\dfrac{1}{\#(I_{\ast}^{1,k})}
\sum_{w^{\ast}\in I_{\ast}^{1,k}}\mathcal{N}\big((\xi_{w})_{w^{\ast}}^{1,k}\big)\Bigg)\\
&\leq & \limsup_{k\to\infty}\frac{1}{k}\log\Bigg(\dfrac{1}{\#(I_{\ast}^{1,k})}
\sum_{w^{\ast}\in I_{\ast}^{1,k}}(m+1)!\ .\ (k+1)^{m}\ .\ (\#(\zeta_{w}))^{k+1}\Bigg)\\
&=& \limsup_{k\to\infty}\frac{1}{k}\log\Big((m+1)!\ .\ (k+1)^{m}\ .\ (\#(\zeta_{w}))^{k+1}\Big)=\log(\#(\zeta_{w})).
\end{eqnarray*}

Now we are ready to finish the proof. The equicontinuity assumption of the NAIFS $(X,\Phi)$ implies that $h_{\text{top}}(X,\Phi)=\dfrac{1}{n}h_{\text{top}}(X,\Phi^{n})$ for each $n\geq 1$, see Lemma \ref{theorem4}.
Also, by the definition of topological entropy it follows
that for any $\epsilon>0$ there is an open cover $\mathcal{A}$ of $X$
with $h_{\text{top}}(X,\Phi^{n})<h(X,\Phi^{n};\mathcal{A})+\epsilon$. Using these facts and relation (\ref{n55}), we get that for any positive integer $n$ and $\epsilon>0$ there
is an open cover $\mathcal{A}$ of $X$ with

\begin{eqnarray*}
h_{\text{top}}(X,\Phi)
&=& \dfrac{1}{n}h_{\text{top}}(X,\Phi^{n})<\dfrac{1}{n}h(X,\Phi^{n};\mathcal{A})+\dfrac{\epsilon}{n}
\leq\dfrac{1}{n}h(X,\Phi^{n};\mathcal{A}_{w}^{1,n})+\dfrac{\epsilon}{n}\\
&\leq & \dfrac{1}{n}h(X,\Phi^{n};\xi_{w})+\dfrac{\epsilon}{n}\leq\dfrac{1}{n}\log(\#(\zeta_{w}))+\dfrac{\epsilon}{n}=
\dfrac{1}{n}\log\mathcal{N}\big(\mathcal{A}_{w}^{1,n}|_{\Omega(\Phi)}\big)+\dfrac{\epsilon}{n},
\end{eqnarray*}
where $w\in I^{1,n}$ is arbitrary. Thus,
\begin{equation*}
h_{\text{top}}(X,\Phi)\leq\dfrac{1}{n}\log\Bigg(\dfrac{1}{\#(I^{1,n})}\sum_{w\in I^{1,n}}\mathcal{N}(\mathcal{A}_{w}^{1,n}|_{\Omega(\Phi)})\Bigg)+\dfrac{\epsilon}{n}.
\end{equation*}
Taking the upper limit when $n\to\infty$, we have
\begin{equation*}
h_{\text{top}}(X,\Phi)\leq h(\Omega(\Phi),\Phi|_{\Omega(\Phi)};\mathcal{A})\leq
h_{\text{top}}(\Omega(\Phi),\Phi|_{\Omega(\Phi)}),
\end{equation*}
that completes the proof.
\end{proof}
\begin{remark}
The equicontinuity assumption in Theorem \ref{theorem10} is necessary, because in the proof of Theorem \ref{theorem10} we use the equality $h_{\text{top}}(X,\Phi^{n})=n\ .\ h_{\text{top}}(X,\Phi)$ that is not true (in general case) without the equicontinuity assumption, see Remark \ref{remark10} and Lemma \ref{theorem4}.
\end{remark}
\section{Specification property and entropy}\label{section4}
The notion of entropy is one of the most important objects in dynamical systems, either as a topological
invariant or as a measure of the chaoticity of dynamical systems. Several notions of entropy have
been introduced for other branches of dynamical systems in an attempt to describe their dynamical
characteristics. In this section, we define entropy points for NAIFSs. The notion of entropy points was
defined for finitely generated pseudogroup actions, finitely generated semigroup actions and
non-autonomous discrete dynamical systems, respectively in \cite{AB}, \cite{FRPV} and \cite{JNFG}.
Roughly speaking, entropy points are
those that their local neighborhoods reflect the complexity of the entire dynamical system in the context of
topological entropy. Also, we define a notion of specification property for NAIFSs and characterize entropy
points and topological entropy for NAIFSs with the specification property.
\begin{definition}
An NAIFS $(X,\Phi)$ of continuous maps on a compact topological space $X$, admits
an \emph{entropy point} $x_{0}\in X$ if for every open neighbourhood $U$ of $x_{0}$ the equality
$h_{\text{top}}(X, \Phi)=h_{\text{top}}(\text{cl}(U), \Phi)$ holds.
\end{definition}

The notion of specification was first introduced in the 1970s as a property of uniformly hyperbolic basic
pieces and became a characterization of complexity in dynamical systems. Thus, several notions of specification had been introduced in an attempt to describe their dynamical characteristics for dynamical systems \cite{JNFG,FRPV,DJT,VSY1,KY}. In the following definition, we give a concept of specification property for NAIFSs.

\begin{definition}
An NAIFS $(X,\Phi)$ of continuous maps on a compact metric space $(X,d)$, is said to
have \emph{the specification property} if for every $\delta>0$ there is
$N(\delta)\in\mathbb{N}$ such that for each $w\in I^{1,\infty}$, any $x_{1},x_{2},\ldots,x_{s}\in X$
with $s\geq 2$ and any sequence $0=j_{1}\leq k_{1}<j_{2}\leq k_{2}<\cdots<j_{s}\leq k_{s} $ of
integers with $j_{n+1}-k_{n}\geq N(\delta)$ for $n=1,\ldots,s-1$, there is a point $x\in X$ such
that $d(\varphi_{w}^{1,i}(x),\varphi_{w}^{1,i}(x_{m}))\leq\delta$ for each $1\leq m\leq s$ and
any $j_{m}\leq i\leq k_{m}$. In other words, an NAIFS $(X,\Phi)$ has the specification property if we have the specification property along every branch $w\in I^{1,\infty}$ as a non-autonomous discrete dynamical system, where $N(\delta)$ is independent of $w\in I^{1,\infty}$, for each $\delta>0$.
\end{definition}
In the last section, we illustrate some examples of NAIFSs for which the specification property hold.

Rodrigues and Varandas \cite{FRPV} showed that for any finitely generated continuous semigroup action
of local homeomorphisms on a compact Riemannian manifold with the strong orbital specification property
(weak orbital specification property), every point is an entropy point. Also, they showed that any
finitely generated continuous semigroup action on a compact metric space with the
strong orbital specification property (weak orbital specification property under some other
conditions) has positive topological entropy. Also, Nazarian Sarkooh and Ghane \cite{JNFG} showed
that every non-autonomous discrete dynamical system of surjective maps with the specification property
has positive topological entropy and all points are entropy point; in particular, it is topologically chaotic.
In this section, we extend these results to NAIFSs.
\subsection{Specification property and entropy points}
We investigate here the relation between the specification property of NAIFSs and the existence of
entropy points.
\begin{theorem}\label{theorem2}
Let $(X,\Phi)$ be an NAIFS of surjective continuous maps on a compact metric space $(X,d)$ without any isolated point. If the NAIFS $(X,\Phi)$ satisfies the specification property, then every point of $X$ is an entropy point.
\end{theorem}
\begin{proof}
According to Lemma \ref{theorem1}, $h_{\text{top}}(X,\Phi)\leq h_{\text{top}}(X,\Phi_{k})$ for every
$k\geq 1$. Also, by Lemma \ref{lemma1},
\begin{equation*}
 h_{\text{top}}(X, \Phi)=\lim_{\epsilon\to 0}\limsup_{n\to\infty}\frac{1}{n}\log S_{n}(\epsilon,\Phi),
\ \text{where}\
S_{n}(\epsilon,\Phi)=\dfrac{1}{\#(I^{1,n})}\sum_{w\in I^{1,n}}s_{n}(w,\epsilon,\Phi).
\end{equation*}
Using these facts, we show that for every $z\in X$ and every open neighborhood $V$ of $z$,
$h_{\text{top}}(X,\Phi)=h_{\text{top}}(\text{cl}(V),\Phi)$.
For $\epsilon>0$ define $W_{\epsilon}:=\{y\in V: d(y,\partial V)>\frac{\epsilon}{4}\}$.
Fix $\epsilon>0$ such that the open set $W_{\epsilon}$ is non-empty.
Take $N(\frac{\epsilon}{4})\geq 1$ given by the definition of specification property. Fix
$w=w_{1}w_{2}\ldots w_{N(\frac{\epsilon}{4})}w_{N(\frac{\epsilon}{4})+1}\ldots w_{N(\frac{\epsilon}{4})+n}\in I^{1,N(\frac{\epsilon}{4})+n}$ and take
\begin{equation*}
 w^{\prime}:=w|^{N(\frac{\epsilon}{4})}=w_{N(\frac{\epsilon}{4})+1}w_{N(\frac{\epsilon}{4})+2}\ldots w_{N(\frac{\epsilon}{4})+n}\in I^{N(\frac{\epsilon}{4})+1,n}.
\end{equation*}
 Let
\begin{itemize}
\item
 $E:=\{z_{1}, z_{2}, \ldots, z_{l}\}\subseteq X$ be a
maximal $(n,w^{\prime},\epsilon;\Phi_{N(\frac{\epsilon}{4})+1})$-separated set,
\item
$E^{\prime}=\{z_{1}^{\prime}, z_{2}^{\prime}, \ldots, z_{l}^{\prime}\}\subseteq X$ be a preimage
set of $E$ under $\varphi_{w}^{1,N(\frac{\epsilon}{4})}$, i.e.
$\varphi_{w}^{1,N(\frac{\epsilon}{4})}(z_{i}^{\prime})=z_{i}$ for $1\leq i\leq l$,
\item
$y\in W_{\epsilon}$ be an arbitrary point ($W_{\epsilon}\neq \emptyset$, because $X$ does not have any isolated point).
\end{itemize}

Let $j_{1}=k_{1}=0$, $j_{2}=N(\frac{\epsilon}{4})$ and $k_{2}=N(\frac{\epsilon}{4})+n$. By the definition
of specification property, for every $z_{i}^{\prime}\in E^{\prime}$, by taking $x_{1}=y$ and $x_{2}=z_{i}^{\prime}$, there
exists $y_{i}\in B(y,\frac{\epsilon}{4})$ such that
$\varphi_{w}^{1,N(\frac{\epsilon}{4})}(y_{i})\in B(\varphi_{w}^{1,N(\frac{\epsilon}{4})}(z_{i}^{\prime});w^{\prime},n,\frac{\epsilon}{4})=B(z_{i};w^{\prime},n,\frac{\epsilon}{4})$.
Since $E:=\{z_{1},z_{2},\ldots,z_{l}\}\subseteq X$ is a
maximal $(n,w^{\prime},\epsilon;\Phi_{N(\frac{\epsilon}{4})+1})$-separated set, the set $\{y_{i}\}_{i=1}^{l}\subseteq cl(V)$ is
$(N(\frac{\epsilon}{4})+n,w,\frac{\epsilon}{2};\Phi)$-separated. So
$s_{N(\frac{\epsilon}{4})+n}(cl(V);w,\frac{\epsilon}{2},\Phi)\geq
s_{n}(w^{\prime},\epsilon,\Phi_{N(\frac{\epsilon}{4})+1})$, that implies
\begin{equation*}
S_{N(\frac{\epsilon}{4})+n}(\text{cl}(V);\frac{\epsilon}{2},\Phi)
\geq \#(I^{1,N(\frac{\epsilon}{4})})\ .\ S_{n}(\epsilon,\Phi_{N(\frac{\epsilon}{4})+1})
\geq S_{n}(\epsilon,\Phi_{N(\frac{\epsilon}{4})+1}).
\end{equation*}
Thus
\begin{eqnarray*}
\limsup_{n\to\infty}\dfrac{1}{n}\log S_{n}(\text{cl}(V);\frac{\epsilon}{2},\Phi)
&=& \limsup_{n\to\infty}\dfrac{1}{N(\frac{\epsilon}{4})+n}\log S_{N(\frac{\epsilon}{4})+n}(\text{cl}(V);\frac{\epsilon}{2},\Phi)\\
&\geq & \limsup_{n\to\infty}\dfrac{1}{N(\frac{\epsilon}{4})+n}\log S_{n}(\epsilon,\Phi_{N(\frac{\epsilon}{4})+1})\\
&= & \limsup_{n\to\infty}\dfrac{1}{n}\log S_{n}(\epsilon,\Phi_{N(\frac{\epsilon}{4})+1}).
\end{eqnarray*}
This implies
$h_{\text{top}}(X,\Phi)\geq h_{\text{top}}(\text{cl}(V),\Phi)\geq h_{\text{top}}(X,\Phi_{N(\frac{\epsilon}{4})+1})\geq h_{\text{top}}(X,\Phi)$.
Hence, we have $h_{\text{top}}(X,\Phi)=h_{\text{top}}(\text{cl}(V),\Phi)$, i.e. every point is an entropy point.
\end{proof}
By Lemma \ref{theorem1} and the proof of Theorem \ref{theorem2}, we conclude the following corollary.
\begin{corollary}\label{corollary1}
Let $(X,\Phi)$ be an NAIFS of surjective continuous maps on a compact metric space $(X,d)$ without any isolated point. If the NAIFS $(X,\Phi)$ satisfies the specification property, then $h_{\text{top}}(X,\Phi)=h_{\text{top}}(X,\Phi_{i})$ for every $i\geq1$.
\end{corollary}
\subsection{Specification property and positive topological entropy}
In this subsection, we show that the specification property is enough to guarantee that any NAIFS
of surjective maps has positive topological entropy. More precisely, we have the following theorem.
\begin{theorem}\label{theorem3}
Let $(X,\Phi)$ be an NAIFS of surjective continuous maps on a compact metric space $(X,d)$ without any isolated point. If the NAIFS $(X,\Phi)$ satisfies the specification property, then it has positive topological entropy, i.e. $h_{\text{top}}(X,\Phi)>0$.
\end{theorem}
\begin{proof}
By Lemma \ref{lemma1}, we know that
\begin{equation*}
 h_{\text{top}}(X, \Phi)=\lim_{\epsilon\to 0}\limsup_{n\to\infty}\frac{1}{n}\log S_{n}(\epsilon,\Phi),
\ \text{where}\
 S_{n}(\epsilon,\Phi)=\dfrac{1}{\#(I^{1,n})}\sum_{w\in I^{1,n}}s_{n}(w,\epsilon,\Phi)
\end{equation*}
and the limit can be replaced by $\sup_{\epsilon>0}$. Thus, it is enough
to prove that there exists $\epsilon>0$ small enough so that
\begin{equation*}
 \limsup_{n\to\infty}\frac{1}{n}\log S_{n}(\epsilon,\Phi)>0.
\end{equation*}
Let $\epsilon>0$ be small and fixed so that there are at least two distinct $2\epsilon$-separated
points $x_{1},y_{1}\in X$, i.e. $d(x_{1},y_{1})>2\epsilon$ (note that X has no any isolated point). Let $N(\frac{\epsilon}{2})\geq 1$ be given by
the definition of specification property.

Fix $w\in I^{1,N(\frac{\epsilon}{2})}$. Take $j_{1}=k_{1}=0$, $j_{2}=k_{2}=N(\frac{\epsilon}{2})$ and
consider preimages $x_{2}$ of $x_{1}$ and $y_{2}$ of $y_{1}$
under $\varphi_{w}^{1,N(\frac{\epsilon}{2})}$, i.e. $\varphi_{w}^{1,N(\frac{\epsilon}{2})}(x_{2})=x_{1}$
and $\varphi_{w}^{1,N(\frac{\epsilon}{2})}(y_{2})=y_{1}$. By applying the specification property to pairs $(x_{1},x_{2})$, $(x_{1},y_{2})$, $(y_{1},x_{2})$ and $(y_{1},y_{2})$,
there are $x_{1,1},x_{1,2}\in B(x_{1},\frac{\epsilon}{2})$ and
$y_{1,1},y_{1,2}\in B(y_{1},\frac{\epsilon}{2})$ such that
\begin{equation*}
\varphi_{w}^{1,N(\frac{\epsilon}{2})}(x_{1,1}),\varphi_{w}^{1,N(\frac{\epsilon}{2})}(y_{1,2})\in B(x_{1},\frac{\epsilon}{2})
\ \ \text{and}\ \
\varphi_{w}^{1,N(\frac{\epsilon}{2})}(x_{1,2}),\varphi_{w}^{1,N(\frac{\epsilon}{2})}(y_{1,1})\in B(y_{1},\frac{\epsilon}{2}).
\end{equation*}
It is clear that the set $\{x_{1,1},x_{1,2},y_{1,1},y_{1,2}\}$
is $(N(\frac{\epsilon}{2}),w,\epsilon;\Phi)$-separated. In particular, it follows
that $s_{N(\frac{\epsilon}{2})}(w,\epsilon,\Phi)\geq 2^{2}$. Hence, we have
\begin{equation*}
 S_{N(\frac{\epsilon}{2})}(\epsilon,\Phi)=\dfrac{1}{\#(I^{1,N(\frac{\epsilon}{2})})}
 \sum_{w\in I^{1,N(\frac{\epsilon}{2})}}s_{N(\frac{\epsilon}{2})}(w,\epsilon,\Phi)\geq\dfrac{1}{\#(I^{1,N(\frac{\epsilon}{2})})}\sum_{w\in I^{1,N(\frac{\epsilon}{2})}}2^{2}=2^{2}.
\end{equation*}
Fix $w\in I^{1,2N(\frac{\epsilon}{2})}$. Take $j_{1}=k_{1}=0$, $j_{2}=k_{2}=N(\frac{\epsilon}{2})$
and $j_{3}=k_{3}=2N(\frac{\epsilon}{2})$. Consider preimages $x_{2}$ of $x_{1}$ and $y_{2}$ of
$y_{1}$ under $\varphi_{w}^{1,N(\frac{\epsilon}{2})}$, i.e.
$\varphi_{w}^{1,N(\frac{\epsilon}{2})}(x_{2})=x_{1}$ and
$\varphi_{w}^{1,N(\frac{\epsilon}{2})}(y_{2})=y_{1}$.
Also, consider preimages $x_{3}$ of $x_{1}$ and $y_{3}$ of $y_{1}$ under
$\varphi_{w}^{1,2N(\frac{\epsilon}{2})}$, i.e. $\varphi_{w}^{1,2N(\frac{\epsilon}{2})}(x_{3})=x_{1}$
and $\varphi_{w}^{1,2N(\frac{\epsilon}{2})}(y_{3})=y_{1}$. By applying the specification property to triples $(x_{1},x_{2},x_{3})$, $(x_{1},x_{2},y_{3})$, $(x_{1},y_{2},x_{3})$, $(x_{1},y_{2},y_{3})$, $(y_{1},y_{2},y_{3})$, $(y_{1},y_{2},x_{3})$, $(y_{1},x_{2},y_{3})$ and $(y_{1},x_{2},x_{3})$,
there are $x_{1,1},x_{1,2},x_{1,3},x_{1,4}\in B(x_{1},\frac{\epsilon}{2})$ and
$y_{1,1},y_{1,2},y_{1,3},y_{1,4}\in B(y_{1},\frac{\epsilon}{2})$ such that
\begin{itemize}
\item
$\varphi_{w}^{1,N(\frac{\epsilon}{2})}(x_{1,1}),\varphi_{w}^{1,2N(\frac{\epsilon}{2})}(x_{1,1})\in B(x_{1},\frac{\epsilon}{2})\ \ \text{and}\ \ \varphi_{w}^{1,N(\frac{\epsilon}{2})}(x_{1,4}),\varphi_{w}^{1,2N(\frac{\epsilon}{2})}(x_{1,4})\in B(y_{1},\frac{\epsilon}{2});$
\item
$\varphi_{w}^{1,N(\frac{\epsilon}{2})}(x_{1,2})\in B(x_{1},\frac{\epsilon}{2})\ \ \text{and}\ \  \varphi_{w}^{1,2N(\frac{\epsilon}{2})}(x_{1,2})\in B(y_{1},\frac{\epsilon}{2});$
\item
$\varphi_{w}^{1,N(\frac{\epsilon}{2})}(x_{1,3})\in B(y_{1},\frac{\epsilon}{2})\ \ \text{and}\ \  \varphi_{w}^{1,2N(\frac{\epsilon}{2})}(x_{1,3})\in B(x_{1},\frac{\epsilon}{2});$
\item
$\varphi_{w}^{1,N(\frac{\epsilon}{2})}(y_{1,1}),\varphi_{w}^{1,2N(\frac{\epsilon}{2})}(y_{1,1})\in B(y_{1},\frac{\epsilon}{2})\ \ \text{and}\ \ \varphi_{w}^{1,N(\frac{\epsilon}{2})}(y_{1,4}),\varphi_{w}^{1,2N(\frac{\epsilon}{2})}(y_{1,4})\in B(x_{1},\frac{\epsilon}{2});$
\item
$\varphi_{w}^{1,N(\frac{\epsilon}{2})}(y_{1,2})\in B(y_{1},\frac{\epsilon}{2})\ \ \text{and}\ \  \varphi_{w}^{1,2N(\frac{\epsilon}{2})}(y_{1,2})\in B(x_{1},\frac{\epsilon}{2});$
\item
$\varphi_{w}^{1,N(\frac{\epsilon}{2})}(y_{1,3})\in B(x_{1},\frac{\epsilon}{2})\ \ \text{and}\ \  \varphi_{w}^{1,2N(\frac{\epsilon}{2})}(y_{1,3})\in B(y_{1},\frac{\epsilon}{2}).$
\end{itemize}
It is clear that the set $\{x_{1,1},x_{1,2},x_{1,3},x_{1,4},y_{1,1},y_{1,2},y_{1,3},y_{1,4}\}$
is $(2N(\frac{\epsilon}{2}),w,\epsilon;\Phi)$-separated. In particular, it follows
that $s_{2N(\frac{\epsilon}{2})}(w,\epsilon,\Phi)\geq 2^{3}$. Hence, we have
\begin{equation*}
 S_{2N(\frac{\epsilon}{2})}(\epsilon,\Phi)=\dfrac{1}{\#(I^{1,2N(\frac{\epsilon}{2})})}
 \sum_{w\in I^{1,2N(\frac{\epsilon}{2})}}s_{2N(\frac{\epsilon}{2})}(w,\epsilon,\Phi)\geq\dfrac{1}{\#(I^{1,2N(\frac{\epsilon}{2})})}\sum_{w\in I^{1,2N(\frac{\epsilon}{2})}}2^{3}=2^{3}.
\end{equation*}
Now, fix $w\in I^{1,dN(\frac{\epsilon}{2})}$ where $d\in\mathbb{N}$. Taking
$j_{1}=k_{1}=0,\ j_{2}=k_{2}=N(\frac{\epsilon}{2}),\ j_{3}=k_{3}=2N(\frac{\epsilon}{2}),\ \ldots,\ j_{d}=k_{d}=(d-1)N(\frac{\epsilon}{2}),\ j_{d+1}=k_{d+1}=dN(\frac{\epsilon}{2})$
and consider the preimages $x_{i}$ of $x_{1}$ and $y_{i}$ of $y_{1}$ under
$\varphi_{w}^{1,(i-1)N(\frac{\epsilon}{2})}$ for $i=2,\ldots,d+1$, i.e.
$\varphi_{w}^{1,(i-1)N(\frac{\epsilon}{2})}(x_{i})=x_{1}$ and
$\varphi_{w}^{1,(i-1)N(\frac{\epsilon}{2})}(y_{i})=y_{1}$. By repeating the previous reasoning
for $(d+1)$-tuples in which the $i$th component choosing from the set $\{x_{i},y_{i}\}$, it follows that
$s_{dN(\frac{\epsilon}{2})}(w,\epsilon,\Phi)\geq 2^{d+1}$. By taking summation over
$w\in I^{1,dN(\frac{\epsilon}{2})}$, we have
\begin{equation*}
S_{dN(\frac{\epsilon}{2})}(\epsilon,\Phi)=\dfrac{1}{\#(I^{1,dN(\frac{\epsilon}{2})})}
\sum_{w\in I^{1,dN(\frac{\epsilon}{2})}}s_{dN(\frac{\epsilon}{2})}(w,\epsilon,\Phi)\geq\dfrac{1}{\#(I^{1,dN(\frac{\epsilon}{2})})}\sum_{w\in I^{1,dN(\frac{\epsilon}{2})}}2^{d+1}=2^{d+1}.
\end{equation*}
Hence,
\begin{equation*}
\limsup_{n\to\infty}\frac{1}{n}\log S_{n}(\epsilon,\Phi)
\geq\limsup_{d\to\infty}\frac{1}{dN(\frac{\epsilon}{2})}\log
S_{dN(\frac{\epsilon}{2})}(\epsilon,\Phi)\geq \limsup_{d\to\infty}\frac{1}{dN(\frac{\epsilon}{2})}\log
2^{d+1}=\dfrac{\log 2}{N(\frac{\epsilon}{2})}.
\end{equation*}
This proves that the topological entropy is positive and finishes the proof.
\end{proof}
As a direct consequence of Theorem \ref{theorem3} and Lemma \ref{theorem1} we have the following corollary.
\begin{corollary}\label{corollary2}
Let $(X,\Phi)$ be an NAIFS of surjective continuous maps on a compact metric space $(X,d)$ without any isolated point. If the NAIFS $(X,\Phi)$ satisfies the specification property, then it has positive asymptotical topological entropy. In particular, the NAIFS $(X,\Phi)$ is topologically chaotic.
\end{corollary}
In Theorem \ref{theorem2}, we show that for surjective NAIFSs with the specification property, local
neighborhoods reflect the complexity of the entire dynamical system from the viewpoint of entropy theory.
Also, in Theorem \ref{theorem3} we show that surjective NAIFSs with the specification property have positive topological entropy. Hence, by Theorem \ref{theorem2}, local neighborhoods have positive topological entropy.
More precisely, we have the following corollary.
\begin{corollary}
Let $(X,\Phi)$ be an NAIFS of surjective continuous maps on a compact metric space $(X,d)$ without any isolated point. If the NAIFS $(X,\Phi)$ satisfies the specification property, then $h_{\text{top}}(\text{cl}(V),\Phi)>0$ for any $x\in X$ and any open neighborhood $V$ of $x$.
\end{corollary}
\section{Topological pressure}\label{section5}
The notion of topological pressure that is a fundamental notion in thermodynamic formalism is a generalization of topological entropy for dynamical
systems \cite{PW}. Topological pressure is the
main tool in studying dimension of invariant sets and measures for dynamical systems in dimension theory.
Our purpose in this section is to introduce and study the notion of topological pressure for NAIFSs on a
compact topological space.

Consider an NAIFS $(X,\Phi)$ of continuous maps on a compact metric space $(X,d)$.
Let $C(X,\mathbb{R})$ be the space of real-valued continuous functions of $X$. For $\psi\in C(X,\mathbb{R})$
and finite word $w\in I^{m,n}$ we denote $\Sigma_{j=0}^{n}\psi(\varphi_{w}^{m,j})(x)$
by $S_{w,n}\psi(x)$. Also, for subset $U$ of $X$ we put $S_{w,n}\psi(U)=\sup_{x\in U}S_{w,n}\psi(x)$.
\subsection{Definition of topological pressure using spanning sets}
For $\epsilon>0$, $n\geq 1$, $w\in I^{1,n}$ and $\psi\in C(X,\mathbb{R})$, put
\begin{equation*}
Q_{n}(\Phi;w,\psi,\epsilon):=\inf_{F}\Bigg\{\sum_{x\in F}e^{S_{w,n}\psi(x)}: F\ \text{is a}\ (w,n,\epsilon;\Phi)\text{-spanning set for}\ X\Bigg\}
\end{equation*}
and taking
\begin{equation*}
  Q_{n}(\Phi;\psi,\epsilon):=\dfrac{1}{\#(I^{1,n})}\sum_{w\in I^{1,n}}Q_{n}(\Phi;w,\psi,\epsilon).
\end{equation*}

\begin{remark}\label{remark1}
By definitions the following statements are true.
\begin{enumerate}
\item[(1)]
$0<Q_{n}(\Phi;w,\psi,\epsilon)\leq\|e^{S_{w,n}\psi}\| r_{n}(w,\epsilon,\Phi)<\infty$, where
$\|\psi\|=\max_{x\in X}|\psi(x)|$. Hence,
$0<Q_{n}(\Phi;\psi,\epsilon)\leq e^{(n+1)\|\psi\|} R_{n}(\epsilon,\Phi)<\infty$.
\item[(2)]
If $\epsilon_{1}<\epsilon_{2}$, then $Q_{n}(\Phi;w,\psi,\epsilon_{1})\geq Q_{n}(\Phi;w,\psi,\epsilon_{2})$. Hence, $Q_{n}(\Phi;\psi,\epsilon_{1})\geq Q_{n}(\Phi;\psi,\epsilon_{2})$.
\item[(3)]
$Q_{n}(\Phi;w,0,\epsilon)=r_{n}(w,\epsilon,\Phi)$. Hence, $Q_{n}(\Phi;0,\epsilon)=R_{n}(\epsilon,\Phi)$.
\item[(4)]
In the definition of $Q_{n}(\Phi;w,\psi,\epsilon)$, it suffices to take the infinium over those spanning sets which do not have proper subsets that $(w,n,\epsilon;\Phi)$-span $X$. This is because $e^{S_{w,n}\psi(x)}>0$.
\end{enumerate}
\end{remark}
Set
$$Q(\Phi;\psi,\epsilon):=\limsup_{n\to\infty}\frac{1}{n}\log Q_{n}(\Phi;\psi,\epsilon).$$
\begin{remark}\label{remark2}
By Remark \ref{remark1}, the following two facts hold.
\begin{enumerate}
\item[(1)]
$Q(\Phi;\psi,\epsilon)\leq\|\psi\|+\limsup_{n\to\infty}\frac{1}{n}\log R_{n}(\epsilon,\Phi)<\infty.$
\item[(2)]
If $\epsilon_{1}<\epsilon_{2}$, then $Q(\Phi;\psi,\epsilon_{1})\geq Q(\Phi;\psi,\epsilon_{2})$, i.e.
$Q(\Phi;\psi,\epsilon)$ is non-decreasing with respect to $\epsilon$.
\end{enumerate}
\end{remark}
\begin{definition}
For $\psi\in C(X,\mathbb{R})$, the \emph{topological pressure} of an NAIFS $(X,\Phi)$ with respect
to $\psi$ is defined as
$$P_{\text{top}}(\Phi,\psi):=\lim_{\epsilon\to 0}Q(\Phi;\psi,\epsilon)=\lim_{\epsilon\to 0}\limsup_{n\to\infty}\frac{1}{n}\log Q_{n}(\Phi;\psi,\epsilon).$$
\end{definition}
This is a natural extension of the definition of topological pressure for autonomous dynamical systems,
non-autonomous discrete dynamical systems and semigroup actions. Also, it is clear
that $P_{\text{top}}(\Phi,0)=h_{\text{top}}(X,\Phi)$.
\begin{remark}
Note that by part $(2)$ of Remark \ref{remark2}, the topological pressure $P_{\text{top}}(\Phi,\psi)$ always exists, but it could be infinite. Indeed, assume
$\#(I^{(j)})=1$, $\Phi^{(j)}=\{\varphi\}$ for all $j\geq 1$ which yields an autonomous dynamical system and take the observer $\psi=0$. In this case, we have
$P_{\text{top}}(\Phi,\psi)=h_{\text{top}}(\varphi)$ which is the classical topological entropy in the sense of
Bowen. Thus, if we choose an autonomous system $\varphi$ with the observer $\psi=0$ that possesses infinite topological entropy \cite{DBBB} then $P_{\text{top}}(\Phi,\psi)=\infty$, where the NAIFS $\Phi$
defined as above.
\end{remark}
\subsection{Definition of topological pressure using separated sets}
For $\epsilon>0$, $n\geq 1$, $w\in I^{1,n}$ and $\psi\in C(X,\mathbb{R})$, put
\begin{equation*}
 P_{n}(\Phi;w,\psi,\epsilon):=\sup_{E}\Bigg\{\sum_{x\in E}e^{S_{w,n}\psi(x)}: E\ \text{is a}\ (w,n,\epsilon;\Phi)\text{-separated set for}\ X\Bigg\}
\end{equation*}
and taking
\begin{equation*}
P_{n}(\Phi;\psi,\epsilon):=\dfrac{1}{\#(I^{1,n})}\sum_{w\in I^{1,n}}P_{n}(\Phi;w,\psi,\epsilon).
\end{equation*}
\begin{remark}\label{remark3}
By definitions the following statements are true.
\begin{itemize}
\item[(1)]
If $\epsilon_{1}<\epsilon_{2}$, then $P_{n}(\Phi;w,\psi,\epsilon_{1})\geq P_{n}(\Phi;w,\psi,\epsilon_{2})$.
Hence, $P_{n}(\Phi;\psi,\epsilon_{1})\geq P_{n}(\Phi;\psi,\epsilon_{2})$.
\item[(2)]
$P_{n}(\Phi;w,0,\epsilon)=s_{n}(w,\epsilon,\Phi)$. Hence,
$0<P_{n}(\Phi;0,\epsilon)=S_{n}(\epsilon,\Phi)$.
\item[(3)]
In the definition of $P_{n}(\Phi;w,\psi,\epsilon)$, it suffices to take the supremum over all
$(w,n,\epsilon;\Phi)$-separated sets having maximal cardinality. This is because $e^{S_{w,n}\psi(x)}>0$.
\item[(4)]
$Q_{n}(\Phi;\psi,\epsilon)\leq P_{n}(\Phi;\psi,\epsilon)$.
\begin{proof}
Fix $w\in I^{1,n}$. Since $e^{S_{w,n}\psi(x)}>0$ and by the fact that each $(w,n,\epsilon;\Phi)$-separated set which cannot be enlarged to another $(w,n,\epsilon;\Phi)$-separated set must be a $(w,n,\epsilon;\Phi)$-spanning set for $X$, we have $Q_{n}(\Phi;w,\psi,\epsilon)\leq P_{n}(\Phi;w,\psi,\epsilon)$. Hence, by the definition of $Q_{n}(\Phi;\psi,\epsilon)$ and $P_{n}(\Phi;\psi,\epsilon)$, we have
$Q_{n}(\Phi;\psi,\epsilon)\leq P_{n}(\Phi;\psi,\epsilon)$.
\end{proof}
\item[(5)]
If $\delta=\sup\{|\psi(x)-\psi(y)|:d(x,y)<\frac{\epsilon}{2}\}$, then
$P_{n}(\Phi;\psi,\epsilon)\leq e^{(n+1)\delta}Q_{n}(\Phi;\psi,\frac{\epsilon}{2})$.
\begin{proof}
Fix $w\in I^{1,n}$. Let $E$ be a $(w,n,\epsilon;\Phi)$-separated set and $F$ is a $(w,n,\frac{\epsilon}{2};\Phi)$-spanning set. Define $\phi:E\to F$ by choosing, for each $x\in E$, some point $\phi(x)\in F$ with $d_{w,n}(x,\phi(x))<\frac{\epsilon}{2}$. The point $\phi(x)\in F$ that satisfies in this condition is unique. Then $\phi$ is injective and
\begin{eqnarray*}
\sum_{y\in F}e^{S_{w,n}\psi(y)}\geq\sum_{y\in \phi(E)}e^{S_{w,n}\psi(y)}
&\geq & \Big(\min_{x\in E}e^{S_{w,n}\psi(\phi(x))-S_{w,n}\psi(x)}\Big)\sum_{x\in E}e^{S_{w,n}\psi(x)}\\
&\geq & e^{-(n+1)\delta}\sum_{x\in E}e^{S_{w,n}\psi(x)}.
\end{eqnarray*}
Therefore $P_{n}(\Phi;w,\psi,\epsilon)\leq e^{(n+1)\delta}Q_{n}(\Phi;w,\psi,\frac{\epsilon}{2})$. Hence, by the definition of $Q_{n}(\Phi;\psi,\frac{\epsilon}{2})$ and $P_{n}(\Phi;\psi,\epsilon)$, we have
$P_{n}(\Phi;\psi,\epsilon)\leq e^{(n+1)\delta}Q_{n}(\Phi;\psi,\frac{\epsilon}{2})$.
\end{proof}
\end{itemize}
\end{remark}
Then, set
\begin{equation*}
 P(\Phi;\psi,\epsilon):=\limsup_{n\to\infty}\frac{1}{n}\log P_{n}(\Phi;\psi,\epsilon).
\end{equation*}
\begin{remark}\label{remark4}
As above, the following statements are true.
\begin{itemize}
\item[(1)]
$Q(\Phi;\psi,\epsilon)\leq P(\Phi;\psi,\epsilon)$, by part $(4)$ of Remark \ref{remark3}.
\item[(2)]
If $\delta=\sup\{|\psi(x)-\psi(y)|:d(x,y)<\frac{\epsilon}{2}\}$, then
$P(\Phi;\psi,\epsilon)\leq \delta+Q(\Phi;\psi,\frac{\epsilon}{2})$, by part $(5)$ of Remark \ref{remark3}.
\item[(3)]
If $\epsilon_{1}<\epsilon_{2}$, then $P(\Phi;\psi,\epsilon_{1})\geq P(\Phi;\psi,\epsilon_{2})$, by part $(1)$ of Remark \ref{remark3}.
\end{itemize}
\end{remark}
\begin{theorem}\label{theorem7}
If $\psi\in C(X,\mathbb{R})$ then $P_{\text{top}}(\Phi,\psi)=\lim_{\epsilon\to 0}P(\Phi;\psi,\epsilon)$.
\end{theorem}
\begin{proof}
The limit exists by part $(3)$ of Remark \ref{remark4}. By part $(1)$ of Remark \ref{remark4}, we have
$P_{\text{top}}(\Phi,\psi)\leq\lim_{\epsilon\to 0}P(\Phi;\psi,\epsilon)$. Also, by part $(2)$ of Remark \ref{remark4}, for any $\delta>0$, we have
$\lim_{\epsilon\to 0}P(\Phi;\psi,\epsilon)\leq\delta+P_{\text{top}}(\Phi,\psi)$, which implies
$\lim_{\epsilon\to 0}P(\Phi;\psi,\epsilon)\leq P_{\text{top}}(\Phi,\psi)$.
Hence, $P_{\text{top}}(\Phi,\psi)=\lim_{\epsilon\to 0}P(\Phi;\psi,\epsilon)$. The proof is completed.
\end{proof}
\subsection{Definition of topological pressure using open covers}
In this subsection we introduce a special class of continuous potentials and provide a formula via open
covers to compute the topological pressure of an NAIFS respect to this class of continuous potentials.
Let $(X, \Phi)$ be an NAIFS of continuous maps on a compact metric space $(X,d)$. Given $\epsilon>0$
and $w\in I^{m,n}$, we say that an open cover $\mathcal{U}$ of $X$ is a $(w,n,\epsilon)$-\emph{cover} if
any open set $U\in\mathcal{U}$ has $d_{w,n}$-diameter smaller than $\epsilon$, where $d_{w,n}$ is the
Bowen-metric introduced in (\ref{eq18}).
To obtain another characterization of the topological pressure using open covers, we need
continuous potentials satisfying a regularity condition. Given $\epsilon>0$, $w\in I^{m,n}$
and $\psi\in\text{C}(X,\mathbb{R})$ we define the \emph{variation} of $S_{w,n}\psi$ on dynamical
balls of radius $\epsilon$ (see (\ref{eq14})) alongside the word $w$ by
\begin{equation*}
\text{Var}_{w,n}(\psi,\epsilon):=\sup_{d_{w,n}(x,y)<\epsilon}|S_{w,n}\psi(x)-S_{w,n}\psi(y)|.
\end{equation*}

We say that potential $\psi$ has \emph{uniform bounded variation on dynamical balls of radius} $\epsilon$ if
there exists $C>0$ so that
\begin{equation*}
\sup_{n\geq 1, w\in I^{1,n}}\text{Var}_{w,n}(\psi,\epsilon)\leq C.
\end{equation*}
The potential $\psi$ has \emph{the uniformly bounded variation property} whenever there exists $\epsilon>0$
so that $\psi$ has the uniform bounded variation on dynamical balls of radius $\epsilon$.

In the following proposition, we use open covers to provide a formula for computation the topological pressure
of an NAIFS respect to this class of continuous potentials.
\begin{proposition}\label{proposition3}
Let $(X,\Phi)$ be an NAIFS of continuous maps on a compact metric space $(X,d)$ and
$\psi:X\to\mathbb{R}$ be a continuous potential with the uniformly bounded variation property. Then,
\begin{equation*}
 P_{\text{top}}(\Phi,\psi)=\lim_{\epsilon\to 0}\limsup_{n\to\infty}\frac{1}{n}\log \Bigg(\dfrac{1}{\#(I^{1,n})}\sum_{w\in I^{1,n}}\inf_{\mathcal{U}}\sum_{U\in\mathcal{U}}e^{S_{w,n}\psi(U)}\Bigg)
\end{equation*}
where the infimum is taken over all open covers $\mathcal{U}$ of $X$ such that $\mathcal{U}$ is
a $(w,n,\epsilon)$-cover.
\end{proposition}
\begin{proof}
By Theorem \ref{theorem7} we know that
$P_{\text{top}}(\Phi,\psi)=\lim_{\epsilon\to 0}\limsup_{n\to\infty}\frac{1}{n}\log P_{n}(\Phi;\psi,\epsilon)$,
where
\begin{equation*}
 P_{n}(\Phi;\psi,\epsilon)=\dfrac{1}{\#(I^{1,n})}\sum_{w\in I^{1,n}}P_{n}(\Phi;w,\psi,\epsilon)=\dfrac{1}{\#(I^{1,n})}\sum_{w\in I^{1,n}}\sup_{E}\sum_{x\in E}e^{S_{w,n}\psi(x)}
\end{equation*}
and the supremum is taken over all sets $E$ that are $(w,n,\epsilon;\Phi)$-separated. For simplicity,
we denote
\begin{equation*}
C_{n}(\Phi;w,\psi,\epsilon):=\inf_{\mathcal{U}}
\sum_{U\in\mathcal{U}}e^{S_{w,n}\psi(U)}\ \ \text{and}\ \  C_{n}(\Phi;\psi,\epsilon):=\dfrac{1}{\#(I^{1,n})}\sum_{w\in I^{1,n}}C_{n}(\Phi;w,\psi,\epsilon),
\end{equation*}
where the infimum is taken over all open covers $\mathcal{U}$ of $X$ such that $\mathcal{U}$ is
a $(w,n,\epsilon)$-cover.

Take $\epsilon>0$ and $w\in I^{1,n}$. Given a $(w,n,\epsilon;\Phi)$-maximal
separated set $E$, it follows that $\mathcal{U}:=\{B(x;w,n,\epsilon)\}_{x\in E}$ is a
$(w,n,2\epsilon)$-cover. By the uniformly bounded variation property we have
\begin{equation*}
S_{w,n}\psi(B(x;w,n,\epsilon))=\sup_{z\in B(x;w,n,\epsilon)}S_{w,n}\psi(z)\leq S_{w,n}\psi(x)+C
\end{equation*}
for some constant $C>0$, depending only on $\epsilon$. Consequently, we have
\begin{equation}\label{eq20}
\limsup_{n\to\infty}\frac{1}{n}\log C_{n}(\Phi;\psi,2\epsilon)\leq\limsup_{n\to\infty}\frac{1}{n}
\log P_{n}(\Phi;\psi,\epsilon).
\end{equation}

On the other hand, if $\mathcal{U}$ is $(w,n,\epsilon)$-cover of $X$, then for
any $(w,n,\epsilon;\Phi)$-separated set $E$ we have
that $\mathcal{N}(E)\leq\mathcal{N}(\mathcal{U})$, since the diameter of any $U\in\mathcal{U}$ in
the metric $d_{w,n}$ is less than $\epsilon$. By the uniformly bounded variation property, we have
\begin{equation}\label{eq21}
\limsup_{n\to\infty}\frac{1}{n}\log P_{n}(\Phi;\psi,\epsilon)\leq\limsup_{n\to\infty}\frac{1}{n}
\log C_{n}(\Phi;\psi,\epsilon).
\end{equation}
Now, combining equations (\ref{eq20}) and (\ref{eq21}), we get that
\begin{equation*}
\limsup_{n\to\infty}\frac{1}{n}\log P_{n}(\Phi;\psi,\epsilon)
\leq\limsup_{n\to\infty}\frac{1}{n}\log C_{n}(\Phi;\psi,\epsilon)
\leq\limsup_{n\to\infty}\frac{1}{n}\log P_{n}(\Phi;\psi,\frac{\epsilon}{2}),
\end{equation*}
this completes the proof.
\end{proof}
\subsection{The topological pressure of $\ast$-expansive NAIFSs}
In this subsection, we will be mostly interested in providing conditions to compute the topological pressure
of an NAIFS as a limit at a definite size scale. Hence, we begin with the following definition.
\begin{definition}
Let $(X,\Phi)$ be an NAIFS of continuous maps on a compact metric space $(X,d)$.
For $\delta>0$, the NAIFS $(X, \Phi)$ is said to be $\delta$-\emph{expansive} if for
any $\gamma>0$ and any $x,y\in X$ with $d(x,y)\geq\gamma$, there exists
$k_{0}\geq1$ (depending on $\gamma$) such that $d_{w,n}(x,y)>\delta$ for each $w\in I^{m,n}$
with $n\geq k_{0}$. Also, an NAIFS is said to be $\ast$-expansive if it is $\delta$-expansive for
some $\delta>0$.
\end{definition}
In the next section, we illustrate some examples of NAIFSs which fit in our situation and hence they possess the $\ast$-expansive property.

In the rest of this section, we prove that the  topological pressure of an $\ast$-expansive NAIFS can be
computed as the topological complexity that is observable at a definite size scale. More precisely, we
get the next result.
\begin{theorem}\label{theorem8}
Let $(X,\Phi)$ be a $\delta$-expansive NAIFS of continuous maps on a compact metric space $(X,d)$ for
some $\delta>0$. Then, for every continuous potential $\psi:X\to\mathbb{R}$ and every $0<\epsilon<\delta$,
\begin{equation*}
P_{\text{top}}(\Phi,\psi)=\limsup_{n\to\infty}\frac{1}{n}\log P_{n}(\Phi;\psi,\epsilon)=
\limsup_{n\to\infty}\frac{1}{n}\log\Bigg(\dfrac{1}{\#(I^{1,n})}\sum_{w\in I^{1,n}}\sup_{E}\sum_{x\in E}e^{S_{w,n}\psi(x)}\Bigg),
\end{equation*}
where the supremum is taken over all sets $E$ that are $(w,n,\epsilon;\Phi)$-separated.
\end{theorem}
\begin{proof}
Since $X$ is compact and $\psi:X\to\mathbb{R}$ is continuous, without
loss of generality, we assume that $\psi$ is non-negative. Fix $\gamma$ and $\epsilon$
with $0<\gamma<\epsilon<\delta$.
Then by part $(3)$ of Remark \ref{remark4} it is enough to prove the following inequality
\begin{equation*}
 \limsup_{n\to\infty}\frac{1}{n}\log P_{n}(\Phi;\psi,\gamma)\leq
 \limsup_{n\to\infty}\frac{1}{n}\log P_{n}(\Phi;\psi,\epsilon).
\end{equation*}

By the definition of $\delta$-expansivity, for any two distinct points $x,y\in X$ with $d(x,y)\geq\gamma$,
there exists $k_{0}\geq1$ (depending on $\gamma$) such that $d_{w,n}(x,y)>\delta$ for each
$w\in I^{m,n}$ with $n\geq k_{0}$. Take $w\in I^{1,n+k}$ with $n,k\geq k_{0}$.
Given any $(w|_{n},n,\gamma;\Phi)$-separated set $E$, we claim that the set $E$ is
$(w,n+k,\epsilon;\Phi)$-separated. In fact, given $x,y\in E$ there exists a $0\leq j\leq n$ so that $d(\varphi_{w}^{1,j}(x),\varphi_{w}^{1,j}(y))>\gamma$.
Using that $n+k-j\geq k_{0}$ and the definition of $\delta$-expansivity, it follows
that $d_{w|^{j}, n+k-j}(\varphi_{w}^{1,j}(x),\varphi_{w}^{1,j}(y))>\delta>\epsilon$.
This implies that $d_{w, n+k}(x,y)>\epsilon$. Hence, $E$ is $(w,n+k,\epsilon;\Phi)$-separated,
that prove the claim. Since $\psi$ is non-negative, we have
\begin{equation}\label{eq22}
e^{S_{w,n+k}\psi(x)}=e^{S_{w,n}\psi(x)}e^{S_{w|^{n},k}\psi(\varphi_{w}^{1,n}(x))}\geq
e^{S_{w,n}\psi(x)},
\end{equation}
which implies that $P_{n}(\Phi;\psi,\gamma)\leq P_{n+k}(\Phi;\psi,\epsilon)$ because by
relation (\ref{eq22}) we have
\begin{eqnarray*}
P_{n}(\Phi;\psi,\gamma)
&=&\dfrac{1}{\#(I^{1,n})}\sum_{w\in I^{1,n}}\sup_{E}\sum_{x\in E}e^{S_{w,n}\psi(x)}
=\dfrac{\#(I^{n+1,k})}{\#(I^{1,n+k})}\sum_{w\in I^{1,n}}\sup_{E}\sum_{x\in E}e^{S_{w,n}\psi(x)}\\
&=& \dfrac{1}{\#(I^{1,n+k})}\sum_{w\in I^{1,n+k}}\sup_{E}\sum_{x\in E}e^{S_{w,n}\psi(x)}
\leq\dfrac{1}{\#(I^{1,n+k})}\sum_{w\in I^{1,n+k}}\sup_{E}\sum_{x\in E}e^{S_{w,n+k}\psi(x)}\\
&=& P_{n+k}(\Phi;\psi,\epsilon).
\end{eqnarray*}
Thus,
\begin{equation*}
\limsup_{n\to\infty}\frac{1}{n}\log P_{n}(\Phi;\psi,\gamma)
\leq\limsup_{n\to\infty}\frac{1}{n+k}\log P_{n+k}(\Phi;\psi,\epsilon)\\
\leq\limsup_{n\to\infty}\frac{1}{n}\log P_{n}(\Phi;\psi,\epsilon).
\end{equation*}
This completes the proof.
\end{proof}
\begin{remark}
We observe that in view of the previous characterization given in Proposition \ref{proposition3}, the same
result as Theorem \ref{theorem8} also holds if we consider open covers instead of separated sets. More
precisely, let $(X,\Phi)$ be a $\delta$-expansive NAIFS of continuous maps on a compact metric
space $(X,d)$ for some $\delta>0$.
Then, for every continuous potential $\psi:X\to\mathbb{R}$ with the uniformly bounded variation
property and every $0<\epsilon<\delta$,
\begin{equation*}
P_{\text{top}}(\Phi;\psi)=\limsup_{n\to\infty}\frac{1}{n}\log \Bigg(\dfrac{1}{\#(I^{1,n})}\sum_{w\in I^{1,n}}\inf_{\mathcal{U}}\sum_{U\in\mathcal{U}}e^{S_{w,n}\psi(U)}\Bigg)
\end{equation*}
where the infimum is taken over all open covers $\mathcal{U}$ of $X$ such that $\mathcal{U}$ is
a $(w,n,\epsilon)$-cover.
\end{remark}
\section{Applications}\label{section6}
The main aim of this section is to introduce a special class of NAIFSs having the specifcation and $\ast$-expansive
properties. Rodrigues and Varandas \cite{FRPV}
addressed the specification properties and thermodynamical formalism to deal both
with finitely generated group and semigroup actions. They introduced the notions of specification and orbital specification
properties for the context of group and semigroup actions. Then they
proved that semigroups of expanding maps satisfy the orbital specification properties. We extend this result to uniformly expanding NAIFS.
\begin{definition}
Let $M$ be a compact Riemannian manifold and $f : M \to M$ be a $C^{1}$-local diffeomorphism. We say that $f$ is \emph{expanding} if there exist $\sigma > 1$ and some Riemannian metric on $M$ such that $\|Df(x)v \| \geq \sigma \|v \|$, for every $x \in M$ and every vector $v$ tangent to $M$ at the point $x$.
\end{definition}
We recall the next statement from \cite{VO}. Let $f : M \to M$ be a expanding $C^{1}$-local diffeomorphism on a compact Riemannian manifold $M$. Then, there exist constants $\sigma > 1$ and $\rho > 0$ such that for every $p \in M$ the image of the ball $B(p, \rho)$ contains a neighborhood of the closure of $B(f(p), \rho)$ and $d(f(x), f(y)) \geq \sigma d(x, y)$, for every $x,y\in B(p,\rho)$. Moreover, for any pre-image $x$ of any point $y \in M$, there exists a map $h : B(y, \rho)\to M$ of
class $C^1$ such that $f \circ h = id$, $h(y) = x$ and
\begin{equation}\label{eq9}
 d(h(y_1),h(y_2))\leq \sigma^{-1}d(y_1,y_2)\ \text{for every}\ y_1,y_2 \in B(y,\rho).
\end{equation}
The factors $\sigma$ and $\rho$ will be called the \emph{expansion factor} and \emph{injectivity constant} of the expanding $C^{1}$-local diffeomorphism $f$, respectively. Also the map $h$ is called \emph{inverse branch} of the
$C^{1}$-local diffeomorphism $f$. Inequality (\ref{eq9}) implies that the inverse branches are contractions, with uniform contraction rate $\sigma^{-1}$.

Now, we introduce a class of NAIFSs that will be studied in the present section. Let
$M$ be a compact Riemannian manifold. For any $\sigma>1$ and $\rho>0$, we denote
by $\mathcal{E}(\sigma,\rho)$ the set of all expanding $C^{1}$-local diffeomorphisms on $M$ with expanding factor $\sigma$ and injectivity constant $\rho$.
\begin{definition}\label{def1}
We say that an NAIFS $(M, \Phi)$ is \emph{uniformly expanding} if there
exist $\sigma>1$ and $\rho>0$ such that $\varphi_{i}^{(j)}\in\mathcal{E}(\sigma,\rho)$ for
each $j\in\mathbb{N}$ and $i\in I^{(j)}$. The factors $\sigma$ and $\rho$ will be
called the \emph{uniform expansion factor} and \emph{injectivity constant} of the NAIFS $(M, \Phi)$,
respectively.
\end{definition}
In what follows, we consider a uniformly expanding NAIFS $(M, \Phi)$ with uniform expansion factor $\sigma >1$ and injectivity constant $\rho >0$. By definition, for each $j \in \mathbb{N}$ and $i\in I^{(j)}$, the restriction of $\varphi_{i}^{(j)}$ to each ball $B(x, \rho)$ of radius $\rho$ is injective and its image contains the closure
of $B(\varphi_{i}^{(j)}(x), \rho)$. Thus, the restriction $\varphi_{i}^{(j)}$ to
$B(x,\rho) \cap (\varphi_{i}^{(j)})^{-1}(B(\varphi_{i}^{(j)}(x), \rho))$ is a diffeomorphism onto $B(\varphi_{i}^{(j)}(x), \rho)$.
We denote the inverse branch of $\varphi_{i}^{(j)}$ at $x$ by
\begin{equation*}
 h_{i,x}^{(j)}:B(\varphi_{i}^{(j)}(x), \rho)\to B(x,\rho).
\end{equation*}
It is clear that $h_{i,x}^{(j)}(\varphi_{i}^{(j)}(x)) = x$ and
$\varphi_{i}^{(j)} \circ h_{i,x}^{(j)} =id$. Definition \ref{def1} implies that $h_{i,x}^{(j)}$ is $\sigma^{-1}$-contraction:
\begin{equation}\label{eq13}
d(h_{i,x}^{(j)}(z),h_{i,x}^{(j)}(w))\leq\sigma^{-1}d(z,w)\ \ \text{for every}\ \ z,w\in B(\varphi_i^{(j)}(x), \rho).
\end{equation}

More generally, for finite word $w=w_{m}w_{m+1}\ldots w_{m+n-1}\in I^{m,n}$ with $m,n\geq 1$,
we call the inverse branch of $\varphi_{w}^{m,n}$ at $x$ the composition
\begin{equation*}
 h_{w,x}^{m,n}:= h_{w_{m},x}^{(m)}\circ h_{w_{m+1},\varphi_{w}^{m,1}(x)}^{(m+1)}\circ\cdots\circ h_{w_{m+n-1},\varphi_{w}^{m,n-1}(x)}^{(m+n-1)}: B(\varphi_{w}^{m,n}(x), \rho) \to B(x,\rho).
 \end{equation*}
Observe that $h_{w,x}^{m,n} (\varphi_{w}^{m,n}(x)) = x$ and $\varphi_{w}^{m,n} \circ h_{w,x}^{m,n}= id$. Moreover, for each $0 \leq j \leq n$ we have
\begin{center}
$\varphi_{w}^{m,j} \circ h_{w,x}^{m,n} = h_{w,\varphi_{w}^{m,j}(x)}^{m+j,n-j}\ \ \text{and} \ \ h_{w,\varphi_{w}^{m,j}(x)}^{m+j,n-j}: B(\varphi_{w}^{m,n}(x), \rho) \to B(\varphi_{w}^{m,j}(x),\rho)$
\end{center}
where $h_{w,\varphi_{w}^{m,j}(x)}^{m+j,n-j}:=h_{w_{m+j},\varphi_{w}^{m,j}(x)}^{(m+j)}\circ\cdots\circ h_{w_{m+n-1},\varphi_{w}^{m,n-1}(x)}^{(m+n-1)}$. Hence,
\begin{equation}\label{eq12}
 d(\varphi_{w}^{m,j} \circ h_{w,x}^{m,n} (z), \varphi_{w}^{m,j} \circ h_{w,x}^{m,n} (w)) \leq \sigma^{j-n}d(z,w)
\end{equation}
for every $z,w \in B(\varphi_{w}^{m,n}(x), \rho)$ and every $0 \leq j \leq n$.

In the rest of this section, we show that uniformly expanding NAIFSs satisfy the specifcation and $\ast$-expansive properties. To do this we need the following auxiliary two lemmas.
\begin{lemma}\label{lemma5}
Let $(M, \Phi)$ be a uniformly expanding NAIFS with the uniform expansion factor $\sigma >1$ and injectivity constant $\rho >0$. Then for every $x \in M$, $w\in I^{m,n}$ and $0< \epsilon \leq \rho$ we have
$\varphi_{w}^{m,n}(B(x;w,n,\epsilon))=B(\varphi_{w}^{m,n}(x),\epsilon)$, where $B(x;w,n,\epsilon)$ is the dynamical $(n+1)$-ball with radius $\epsilon$ corresponding to the finite word $w$ around $x$ given by (\ref{eq14}).
\end{lemma}
\begin{proof}
Let $w\in I^{m,n}$ and $B(x;w,n,\epsilon)$ be the dynamical $(n+1)$-ball with radius $\epsilon$ corresponding to the finite word $w$ around $x$. The inclusion $\varphi_{w}^{m,n}(B(x;w,n,\epsilon))\subseteq B(\varphi_{w}^{m,n}(x),\epsilon)$ is an immediate consequence of the definition of a dynamical ball.
To prove the converse, consider the inverse branch
$h_{w,x}^{m,n}: B(\varphi_{w}^{m,n}(x), \rho) \to B(x,\rho)$ of $\varphi_{w}^{m,n}$ at $x$.
Given any $y \in B(\varphi_{w}^{m,n}(x),\epsilon)$, let $z = h_{w,x}^{m,n}(y)$.
Then $\varphi_{w}^{m,n}(z) = y$. By inequality (\ref{eq12}), for $0 \leq j \leq n$, we have
$$d(\varphi_{w}^{m,j}(z), \varphi_{w}^{m,j}(x)) \leq \sigma^{j-n}d(\varphi_{w}^{m,n}(z), \varphi_{w}^{m,n}(x)) \leq d(y, \varphi_{w}^{m,n}(x)) < \epsilon.$$
Hence, $z = h_{w,x}^{m,n}(y)\in B(x;w, n, \epsilon)$ that implies $\varphi_{w}^{m,n}(B(x;w,n,\epsilon))\supseteq B(\varphi_{w}^{m,n}(x),\epsilon)$. This finishes the proof of the lemma.
\end{proof}
The following lemma of the topologically exact property is now folklore and we omit its proof, see \cite[Lemma 18]{FRPV}.
\begin{lemma}\label{lemma4}
Let $(M, \Phi)$ be a uniformly expanding NAIFS on a compact connected Riemannian manifold $M$. Then for any $\delta>0$ there is $N=N(\delta)\in\mathbb{N}$ so that $\varphi_{w}^{m,n}(B(x,\delta))=M$ for every $x\in M$ and $w\in I^{m,n}$ with $n\geq N$.
\end{lemma}
Note that Lemma \ref{lemma4} also implies that each expanding $C^{1}$-local diffeomorphism
on a compact connected Riemannian manifold $M$ is surjective.
\begin{theorem}\label{theorem6}
Let $(M, \Phi)$ be a uniformly expanding NAIFS on a compact connected Riemannian manifold $M$ with the uniform expansion factor $\sigma >1$ and injectivity constant $\rho >0$. Then the NAIFS $(M, \Phi)$ satisfies the specification property.
\end{theorem}
\begin{proof}
The proof of the theorem can be followed from the previous two lemmas. Fix $\delta>0$,
without loss of generality we assume that $\delta<\rho$. Let $w=w_{1}w_{2}\ldots\in I^{1,\infty}$ and $N=N(\delta)$ be given by Lemma \ref{lemma4}. Suppose that points $x_{1},x_{2},\ldots,x_{s}\in M$ with
$s\geq 2$ and sequence $0=j_{1}\leq k_{1}<j_{2}\leq k_{2}<\cdots<j_{s}\leq k_{s}$ of integers with
$j_{n+1}-k_{n}\geq N$ for $n=1,\ldots,s-1$ are given. By Lemma \ref{lemma5} we have
\begin{equation}\label{eq15}
\varphi_{w}^{j_{i}+1,k_{i}-j_{i}}(B(\varphi_{w}^{1,j_{i}}(x_{i});w|^{j_{i}},k_{i}-j_{i},\delta))=B(\varphi_{w}^{j_{i}+1,k_{i}-j_{i}}(\varphi_{w}^{1,j_{i}}(x_{i})),\delta)\ \text{for}\ 1\leq i\leq s.
\end{equation}
Also by Lemma \ref{lemma4} we get
\begin{equation}\label{eq16}
\varphi_{w}^{k_{i}+1,j_{i+1}-k_{i}}(B(\varphi_{w}^{j_{i}+1,k_{i}-j_{i}}(\varphi_{w}^{1,j_{i}}(x_{i})),\delta))=M\ \text{for}\ i=1,\ldots,s-1.
\end{equation}

Equations (\ref{eq15}) and (\ref{eq16}) imply that
for given $\bar{x}_{s}\in B(\varphi_{w}^{1,j_{s}}(x_{s}),w|^{j_{s}},k_{s}-j_{s},\delta)$ we have
$\bar{x}_{s}=\varphi_{w}^{k_{s-1}+1,j_{s}-k_{s-1}}(\tilde{x}_{s-1})$ with
$\tilde{x}_{s-1}\in B(\varphi_{w}^{j_{s-1}+1,k_{s-1}-j_{s-1}}(\varphi_{w}^{1,j_{s-1}}(x_{s-1})),\delta)$,
hence
$\bar{x}_{s}=\varphi_{w}^{k_{s-1}+1,j_{s}-k_{s-1}}\circ\varphi_{w}^{j_{s-1}+1,k_{s-1}-j_{s-1}}(\bar{x}_{s-1})$, for some
\begin{center}
$\bar{x}_{s-1}\in B(\varphi_{w}^{1,j_{s-1}}(x_{s-1});w|^{j_{s-1}},k_{s-1}-j_{s-1},\delta)$.
\end{center}
By repeating this argument, there exists $\bar{x}_{1}\in B(x_{1};w,k_{1},\delta)$, such that for $i=2,\ldots,s$ we have
\begin{equation}\label{eq17}
\bar{x}_{i}=\varphi_{w}^{k_{i-1}+1,j_{i}-k_{i-1}}\circ\varphi_{w}^{j_{i-1}+1,k_{i-1}-j_{i-1}}\circ\cdots\circ\varphi_{w}^{k_{1}+1,j_{2}-k_{1}}\circ\varphi_{w}^{1,k_{1}}(\bar{x}_{1}).
\end{equation}
Now, by equation (\ref{eq17}), $x=\bar{x}_{1}$ satisfies the definition of specification property and finishes the proof of the theorem.
\end{proof}
The next result shows that any uniformly expanding NAIFS satisfies the $\ast$-expansive property.
\begin{proposition}\label{pro6}
Let $(M, \Phi)$ be a uniformly expanding NAIFS with the uniform expansion factor $\sigma >1$ and injectivity constant
$\rho >0$. Then the NAIFS $(M, \Phi)$ is $\ast$-expansive.
\end{proposition}
\begin{proof}
By assumption, all inverse branches of $\varphi_{i}^{(j)}$, for each $j\in\mathbb{N}$ and $i\in I^{(j)}$, are defined on balls of radius $\rho$ and
they are $\sigma^{-1}$-contraction. Take $\delta=\rho$. For given $\gamma>0$, take $k_{0}\geq 1$ (depending on $\gamma$) so that $\sigma^{-k_{0}}\delta<\gamma$. We claim that  for any $x,y\in M$ with $d(x,y)\geq\gamma$ and $w\in I^{m,n}$ with $m\geq 1$ and $n\geq k_{0}$ we have $d_{w,n}(x,y)>\delta$.
Assume, by contradiction, that there exists $w\in I^{m,n}$ with $m\geq 1$ and $n\geq k_{0}$ such that $d_{w,n}(x,y)\leq\delta$. Then, by inequality (\ref{eq12}), we have $d_{w,j}(x,y)\leq \sigma^{j-n}d_{w,n}(x,y)$ for every $0\leq j\leq n$ and so $d(x,y)\leq \sigma^{-n}d_{w,n}(x,y)<\sigma^{-n}\delta\leq\sigma^{-k_{0}}\delta<\gamma$, which is a contradiction. Hence, the NAIFS $(M, \Phi)$ is $\delta$-expansive which completes the proof.
\end{proof}
Now, we illustrate some examples of NAIFSs which fit in our situation.
\begin{example}
Let $\varphi_{A}:\mathbb{T}^{d}\to\mathbb{T}^{d}$ be the linear endomorphism of the torus $\mathbb{T}^{d}=\mathbb{R}^{d}/\mathbb{Z}^{d}$ induced by
some matrix $A$ with integer coefficients and determinant different from zero. Assume that all the
eigenvalues $\lambda_{1},\lambda_{2},\ldots,\lambda_{d}$ of $A$ are larger than $1$ in
absolute value. Then, given any $1<\sigma<\inf_{i}|\lambda_{i}|$, there exists an inner product
in $\mathbb{R}^{d}$ relative to which $||Av||\geq\sigma ||v||$ for every $v\in\mathbb{R}^{d}$. This shows that the transformation $\varphi_{A}$ is expanding, see \cite[Example 11.1.1]{VO}.

Now, let $\mathcal{A}$ be a non-empty finite set of different matrices enjoying the above conditions. Then, each NAIFS $(\mathbb{T}^{d},\Phi)$ consists of the sequence $\{\Phi^{(j)}\}_{j\geq 1}$ of collections $\Phi^{(j)}\subseteq\{\varphi_{A}:A\in\mathcal{A}\}$ is uniformly expanding and by Theorem \ref{theorem6}
and Proposition \ref{pro6} satisfies the specification and $\ast$-expansive properties.
\end{example}
\begin{example}
Let $A$ be a non-empty finite set of positive integers $k>1$ and $S^{1}=\mathbb{R}/\mathbb{Z}$. Consider the set $\mathcal{A}=\{f_{k}:S^{1}\to S^{1}: f_{k}(x)=kx\ \text{(mod 1)},\ k\in A\}$. Then, each NAIFS $(S^{1},\Phi)$ consists of the sequence $\{\Phi^{(j)}\}_{j\geq 1}$ of collections $\Phi^{(j)}\subseteq\mathcal{A}$ is uniformly expanding and by Theorem \ref{theorem6} and Proposition \ref{pro6} satisfies the specification and $\ast$-expansive properties.
\end{example}
\begin{example}
For positive constant $0<\alpha<1$ the Pomeau-Manneville map $\varphi_\alpha:[0,1]\to[0,1]$ given by
\begin{equation*}\label{h}
\varphi_\alpha(x)= \left\{
 \begin{array}{rl}
  x+2^\alpha x^{1+\alpha} & 0 \leq x \leq 1/2\\
  2x-1 & 1/2 < x \leq 1.
 \end{array}\right.
\end{equation*}
Note that, since each Pomeau-Manneville map is semiconjugated to the full shift on two symbols, it satisfies
the specification property (as an autonomous dynamical system), see \cite[Example 3.4]{TBOPV}.
Here, we give an NAIFS $(S^1,\Phi)$ that consists of circle Pomeau-Manneville
maps having the specification property.

Indeed, let us take $0 <\beta<1$ and the family of real numbers
$$\{\alpha_i^{(j)}: 0 <\beta< \alpha_i^{(j)}<1\}_{i \in I^{(j)}}, \ j \in \mathbb{N},$$
where $I^{(j)}$ is a non-empty finite index set for all $j \geq 1$.
Assume $\varphi_i^{(j)}=\varphi_{\alpha_i^{(j)}}$ for all $j\geq 1$ and $i\in I^{(j)}$. We identify the
unit interval $[0, 1]$ with the circle $S^1$, so that the maps become continuous.
Take the NAIFS $(S^1,\Phi)$ consists of the sequence $\{\Phi^{(j)}\}_{j\geq 1}$ of collections $\Phi^{(j)}=\{\varphi_{\alpha_i^{(j)}}\}_{i\in I^{(j)}}$ of Pomeau-Manneville circle maps.
We claim that the NAIFS $(S^1,\Phi)$ satisfies the specification property.
First, we observe that for every $x\in S^{1}$, $\epsilon>0$ and $w\in I^{m,n}$ with $m,n\geq 1$ the dynamical
ball $B(x;w,n,\epsilon)$ satisfes $\varphi_{w}^{m,n}(B(x;w,n,\epsilon))=B(\varphi_{w}^{m,n}(x),\epsilon)$.
Second, although each Pomeau-Manneville map $\varphi_{\alpha_i^{(j)}}$ is not uniformly expanding, it enjoys the following scaling property:
given $\delta > 0$, $\text{diam}(\varphi_{\alpha_i^{(j)}}[0,\delta])\geq \frac{\delta}{2}+\frac{\delta}{2}[1+(1+\beta)\delta^\beta]=c_\delta \text{diam}[0,\delta]$
and $\text{diam}(\varphi_{\alpha_i^{(j)}}(I))\geq \sigma_\delta \text{diam}(I)$ for every ball $I \subset S^1$
of diameter larger or equal to $\delta$, where $\sigma_\delta >1$ (depending on $\delta$) and $c_\delta:=(1+\delta(1+\beta)\delta^\beta)>1$, see \cite{FRPV}. Note that by the choice of the collections $\Phi^{(j)}$ as
above, their derivatives satisfy
$d\varphi_{\alpha_i^{(j)}}(x)\geq (1+(1+\beta)2^\beta x^\beta)\geq (1+(1+\beta)\delta^\beta)$ for every $x \in [\frac{\delta}{2},\frac{1}{2}]$ and $d\varphi_{\alpha_i^{(j)}}(x)=2$
for every $x \in (\frac{1}{2},1]$.
Using the previous expression recursively, we deduce that there exists $N_\delta>0$ such that for each $w \in I^{m,n}$ with $n \geq N_\delta$ one has that $\varphi_w^{m.n}(B(x,\delta))=S^1$, for each $x \in S^1$. This means that the NAIFS $(S^1,\Phi)$ is topologically exact. Thus, we can apply the approach used
in the proof of Theorem \ref{theorem6} to conclude the NAIFS $(S^1,\Phi)$ has the specification property.
\end{example}

In what follows, we give some comments about the specification property of NAIFSs and semigroup (group) actions.

Given a continuous map $g$ on a topological space $X$, we say that $g$ has \emph{finite order} if there exists $n\geq 1$ so that $g^n=id_{X}$.
Let us mention that, in the context of group actions, the existence of elements of generators of finite order is not
an obstruction for the group action to have the specification property in the sense of \cite[Defnition 1]{FRPV} that extends the specifcation property introduced by Ruelle \cite{DR} to more general group actions and differs
from the orbital specification properties which introduced by Rodrigues and Varandas \cite{FRPV} (e.g., the $\mathbb{Z}^2$-action on $\mathbb{T}^2=\mathbb{R}^2 / \mathbb{Z}^2$ whose generators are a hyperbolic automorphism and the reflection on the real axis satisfies the specification property in the sense of
\cite[Defnition 1]{FRPV}, see \cite{FRPV}). However, in the context of NAIFSs, if there
exists $g\in\cap_{j\geq 1}\Phi^{(j)}$ of finite order, then this can not be true, see the following example.
\begin{example}\label{ex22}
Let $(X,\Phi)$ be an NAIFS of continuous maps on a compact metric space $(X,d)$, and let $g:X\to X$ be a continuous map of finite order $n$ such that $g\in\cap_{j\geq 1}\Phi^{(j)}$. We claim that the NAIFS $(X,\Phi)$ does not enjoy the specification property. Assume, by contradiction, that the NAIFS $(X,\Phi)$ satisfies the specification property. Let $\delta>0$ be small and fixed so that there are at least two distinct $2\delta$-separated points $x_{1},x_{2}\in X$, i.e. $d(x_{1},x_{2})>2\delta$. Let $N(\frac{\delta}{2})\geq 1$ be given by the definition of specification property. Then, for the word $w \in I^{1,\infty}$ corresponding to the constant sequence $(g,g,g,\ldots)$ and integers $0=j_{1}=k_{1}<j_{2}=k_{2}$ with $j_{2}-k_{1}=rn\geq N(\frac{\delta}{2})$, for some $r\in\mathbb{N}$, there is a point $x\in X$ such that  $d(x,x_{1})\leq\frac{\delta}{2}$ and
$d(\varphi_{w}^{1,rn}(x),\varphi_{w}^{1,rn}(x_{2}))\leq\frac{\delta}{2}$. Consequently,
\begin{equation*}
\delta<d(x,x_{2})=d(g^{rn}(x),g^{rn}(x_{2}))=d(\varphi_{w}^{1,rn}(x),\varphi_{w}^{1,rn}(x_{2}))\leq\frac{\delta}{2},
\end{equation*}
that is a contradiction.
\end{example}
Note that in \cite{FRPV} the authors introduced three kinds of specification properties for group and semigroup actions: specification property in the sense of Ruelle, strong orbital specification property and weak orbital specification property. For a semigroup action, the claim in Example \ref{ex22} holds whenever we consider the 
strong orbital specification property.


We mention that, an NAIFS generalizes the both concepts of finitely generated semigroup actions and
non-autonomous discrete dynamical systems. The next example shows that the dynamic of an NAIFS differs from semigroup actions.
\begin{example}
Let $f:S^{1}\to S^{1}$ be a $C^{1}$-expanding map of the circle, and let $R_{\alpha}:S^{1}\to S^{1}$ be the rotation of angle $\alpha$. Then, the semigroup $G$ generated by $G_{1}=\{f,R_{\alpha}\}$ does not satisfy the strong orbital specification property, see \cite[Example 31]{FRPV}.

Now, let $(S^{1},\Phi)$ be a uniformly expanding NAIFS with the uniform expansion factor $\sigma >1$ and injectivity constant $\rho >0$. Then, by Theorem \ref{theorem6}, the NAIFS $(S^{1}, \Phi)$ satisfies the specification property. Take $\Psi^{(1)}=\Phi^{(1)}\cup\{R_{\alpha}\}$ and $\Psi^{(j)}=\Phi^{(j)}$ for
all $j\geq 2$. We claim that the NAIFS $(S^{1},\Psi)$ enjoys the specification property. Indeed, let $\delta>0$ be
fixed, without loss of generality we assume that $\delta<\rho$, and take $N(\delta)$ the constant given
by Lemma \ref{lemma4} for the NAIFS $(S^{1},\Phi)$. 
For the NAIFS $(S^{1},\Psi)$, take a word $w=w_{1}w_{2}\ldots\in I^{1,\infty}$, points $x_{1},x_{2},\ldots,x_{s}\in S^{1}$ with $s\geq 2$
and a sequence $0=j_{1}\leq k_{1}<j_{2}\leq k_{2}<\cdots<j_{s}\leq k_{s}$ of integers
with $j_{n+1}-k_{n}\geq N_{\delta}$ for $n=1,\ldots,s-1$, where $N_{\delta}=N(\delta)+1$. By Theorem \ref{theorem6}, if $\psi_{w_{1}}^{(1)}\neq R_{\alpha}$, then there is a point $x\in X$ such that $d(\psi_{w}^{1,i}(x),\psi_{w}^{1,i}(x_{m}))\leq\delta$ for each $1\leq m\leq s$ and
any $j_{m}\leq i\leq k_{m}$. If $\psi_{w_{1}}^{(1)}=R_{\alpha}$, then there is a dynamical $(k_{1}+1)$-ball $B(x_{1};w,k_{1},\epsilon)$ with $\epsilon\leq\delta$ such that $\psi_{w}^{1,k_{1}}(B(x_{1};w,k_{1},\epsilon))=B(\psi_{w}^{1,k_{1}}(x_{1}),\delta)$ (note that, $R_{\alpha}$ is an isometry and $(S^{1},\Phi)$ is a uniformly expanding NAIFS). Now, by the approach used in Theorem \ref{theorem6}, there is a point $x\in S^{1}$ such that
$d(\psi_{w}^{1,i}(x),\psi_{w}^{1,i}(x_{m}))\leq\delta$ for each $1\leq m\leq s$ and
any $j_{m}\leq i\leq k_{m}$. This proves the claim.
\end{example}
\section*{Acknowledgements}
The authors would like to thank the respectful referee for his/her comments on the manuscript.

\end{document}